\documentclass[12pt,a4paper]{article}

\usepackage{amsfonts,amsmath,amsthm,amssymb}
\usepackage{mathrsfs}
\usepackage{bbold}
\usepackage{extarrows}
\usepackage{cite}
\usepackage{authblk}

\expandafter\ifx\csname pdfglyphtounicode\endcsname\relax\else
\expandafter\ifx\csname pdfgentounicode\endcsname\relax\else
\input glyphtounicode.tex
\fi
\fi

\renewcommand{\leq}{\leqslant}

\renewcommand{\geq}{\geqslant}

\makeatletter
\newcommand*{\bigcorr@macro}[2]{\sbox{0}{\mbox{$#1($}}\dimen0=\ht0
                \advance\dimen0 by \dp0
                \multiply\dimen0 by #2 \divide\dimen0 by 100}
\newcommand*{\bigcorr@big}[2]{\mbox{$#1\left#2\bigcorr@macro{#1}{85}\vrule
                   height \dimen0 depth 0pt width 0pt\right.\n@space$}}
\newcommand*{\bigcorr@Big}[2]{\mbox{$#1\left#2\bigcorr@macro{#1}{115}\vrule
                   height \dimen0 depth 0pt width 0pt\right.\n@space$}}
\newcommand*{\bigcorr@bigg}[2]{\mbox{$#1\left#2\bigcorr@macro{#1}{145}\vrule
                   height \dimen0 depth 0pt width 0pt\right.\n@space$}}
\newcommand*{\bigcorr@Bigg}[2]{\mbox{$#1\left#2\bigcorr@macro{#1}{175}\vrule
                   height \dimen0 depth 0pt width 0pt\right.\n@space$}}
\DeclareRobustCommand*{\big}[1]{{\mathpalette\bigcorr@big{#1}}}
\DeclareRobustCommand*{\Big}[1]{{\mathpalette\bigcorr@Big{#1}}}
\DeclareRobustCommand*{\bigg}[1]{{\mathpalette\bigcorr@bigg{#1}}}
\DeclareRobustCommand*{\Bigg}[1]{{\mathpalette\bigcorr@Bigg{#1}}}
\DeclareRobustCommand*{\bi}[1]{{#1}}
\DeclareRobustCommand*{\bil}[1]{\mathopen{\bi{#1}}}

\DeclareRobustCommand*{\bir}[1]{\mathclose{\bi{#1}}}
\makeatother

\makeatletter
\DeclareFontFamily{U}{escorr@fam}{}
\DeclareFontShape{U}{escorr@fam}{m}{n}{%
 <5>sfixed * [3.75] cmsy10%
 <6>sfixed * [4.5] cmsy10%
 <7>sfixed * [5.25] cmsy10%
 <8>sfixed * [6] cmsy10%
 <9>sfixed * [6.75] cmsy10%
 <10>sfixed * [7.5] cmsy10%
 <10.95>sfixed * [8.2125] cmsy10%
 <12>sfixed * [9] cmsy10%
 <14.4>sfixed * [10.8] cmsy10%
 <17.28>sfixed * [12.96] cmsy10%
 <20.74>sfixed * [15.555] cmsy10%
 <24.88>sfixed * [18.66] cmsy10%
 }{}
\DeclareMathAlphabet{\escorr@alph}{U}{escorr@fam}{m}{n}
\newcommand*{\escorr@macro}[2]{\savebox{0}{$#1\escorr@alph{\mathchar"710D}\n@space$}
                               \savebox{1}[\wd0]{\hss$#1\mathchar"002F\n@space$\hss}
                               \dimen0=\ht1 \advance\dimen0by\dp1 \divide\dimen0by16
                               \mbox{\rlap{\usebox{1}}\raisebox{\dimen0}{\usebox{0}}}}
\DeclareRobustCommand{\Memptyset}{{\mathpalette\escorr@macro{}}}
\DeclareRobustCommand{\Mbfemptyset}{{\pmb{\pmb{\Memptyset}}}}
\makeatother

\newtheorem{Thm}{Theorem}
\newtheorem{Prop}{Proposition}
\theoremstyle{remark}
\newtheorem{remark}{Remark}
\newtheorem{example}{Example}

\newcommand{\wb}{\overline}% bar
\newcommand{\diam}{\mathop{\mathrm{diam}}}% diameter
\newcommand{\W}{\mathcal{W}}% the set of finite volumes
\newcommand{\V}{\mathcal{V}}% subset of the set of finite volumes
\newcommand{\Bigsetminus}{\mathbin{\big\backslash}}% Big setminus

\newcommand{\cB}{{\mathscr{B}}}% calligraphic B
\newcommand{\cT}{{\mathscr{T}}}% calligraphic T
\newcommand{\cG}{{\mathscr{G}}}% calligraphic G

\newcommand{\pb}{{\mathrm{P}}}% f.-v. distributions
\newcommand{\Pb}{{\boldsymbol{P}}}% systems of f.-v. distributions
\newcommand{\PB}{{\mathbf{P}}}% random fields

\newcommand{\qb}{{\mathrm{Q}}}% ratios and limits
\newcommand{\Qb}{{\boldsymbol{Q}}}% specifications
\newcommand{\QB}{{\boldsymbol{\mathcal{Q}}}}% canonical specifications

\newcommand{\TE}{{\Delta}}% transition energies on T
\newcommand{\te}{{\delta}}% transition energies in f.-v.
\newcommand{\Te}{{\boldsymbol{\Delta}}}% transition energy fields
\newcommand{\tec}{{\Delta}}% limits of log-ratios
\newcommand{\Tec}{{\boldsymbol{\varDelta}}}% canonical transition energy systems

\newcommand{\hh}{{H}}% f.-v. Hamiltonians
\newcommand{\HH}{{\boldsymbol{H}}}% Hamiltonians
\newcommand{\bPhi}{{\boldsymbol{\Phi}}}% potentials

\begin{document}

\title{On the relationship of energy and probability in models of~classical statistical physics}
\author[1]{S.\ Dachian}
\author[2]{B.S.\ Nahapetian}
\affil[1]{\small University of Lille, CNRS, UMR 8524 --- Laboratoire Paul
  Painlev\'e, 59000 Lille, France}
\affil[1]{\small National Research Tomsk State University, International
  Laboratory of Statistics of Stochastic Processes and Quantitative Finance,
  634050 Tomsk, Russia}
\affil[2]{\small Institute of Mathematics, National Academy of Sciences of the
  Republic of Armenia, 0019 Yerevan, Armenia}
\footnotetext[1]{serguei.dachian@univ-lille.fr}
\footnotetext[2]{nahapet@instmath.sci.am}
\date{}
\maketitle

\begin{abstract}
In this paper we present a new point of view on the mathematical foundations of statistical physics of infinite volume systems.  This viewpoint is based on the newly introduced notions of transition energy function, transition energy field and one-point transition energy field.  The former of them, namely the transition energy function, is a generalization of the notion of relative Hamiltonian introduced by Pirogov and Sinai.  However, unlike the (relative) Hamiltonian, our objects are defined axiomatically by their natural and physically well-founded intrinsic properties. The developed approach allowed us to give a proper mathematical definition of the Hamiltonian without involving the notion of potential, to propose a justification of the Gibbs formula for infinite systems and to answer the problem stated by D.\ Ruelle of how wide the class of specifications, which can be represented in Gibbsian form, is.  Furthermore, this approach establishes a straightforward relationship between the probabilistic notion of (Gibbs) random field and the physical notion of (transition) energy, and so opens the possibility to directly apply probabilistic methods to the mathematical problems of statistical physics.

\bigskip
\noindent
\textbf{Keywords:} random field, specification, 1--specification, Gibbs formula, Hamiltonian, transition energy field, one-point transition energy field.
\end{abstract}

\section*{Introduction}

For the purpose of studying systems of statistical physics with infinite number of particles, Dobrushin in~\cite{D1, D2, D3} and, independently, Lanford and Ruelle in~\cite{LR} introduced the fundamental notion of Gibbs random field (also known as Gibbs state or Gibbs measure) and established its most important properties.  The interest in infinite volume physical systems was mostly prompted by the fact that it was in these systems that the occurrences of phase transitions were mathematically rigorously described.  The above-mentioned works laid the foundations of mathematical statistical physics and predetermined its further impressive development (see, for example, Ruelle~\cite{R,R2}, Sinai~\cite{Sinai}, Malyshev and Minlos~\cite{MM}, Georgii~\cite{Georg}).

In his approach to defining a Gibbs random field (as a probability measure on the space of infinite configurations of a physical system), Dobrushin was considering two families of functions: the specification and the Hamiltonian.  The specification is a family of probability distributions in finite volumes with infinite boundary conditions.  It is characterized by a property of consistency of its elements, which is quite natural from the probabilistic point of view.  As to the Hamiltonian (potential energy), in its standard form it is represented as a sum of finite-dimensional interactions given by some potential.  Based on such a representation of the Hamiltonian and on the Gibbs formula (also known as Boltzmann-Gibbs formula) tying the probability of a state of a physical system to the potential energy of that state, one can construct the main example of a specification: the Gibbsian specification with a given potential.  On this basis, Dobrushin introduced the notion of the Gibbs random field with a given potential, described the structure of the set of such random fields and solved the problems of their existence and uniqueness.

The questions considered in the present paper concern the foundations of the theory of Gibbs random fields.  One of the main outcomes of this work is that many general facts of the Gibbsian theory can be stated in purely probabilistic manner without involving the notion of potential.  This allows to give the theory a classical form: first, general definitions and statements, and then construction of models using representation theorems.

It should be noted that the expression of the Hamiltonian in form of a sum of interactions cannot be considered as its definition, because in case of such an approach, the properties of the Hamiltonian are not postulated, but are instead a consequence of the properties of the potential.  In fact, here we are dealing with a representation of the Hamiltonian in terms of potential.  This representation is quite convenient from the physical point of view, as it allows to easily construct specific models with the desired properties.  Meanwhile, from the mathematical point of view, in order to build a general theory of Gibbs random fields, it is necessary to have a proper definition of the Hamiltonian by means of its intrinsic properties.  This is all the more important, since in statistical physics, many considerations are based on the properties of the Hamiltonian itself (and how these properties are deduced from the properties of the potential is absolutely irrelevant).

The key concepts of our considerations are the newly introduced notions of transition energy function, transition energy field and one-point transition energy field.  Let us note that the former of them, namely the transition energy function, is a generalization of the notion of relative Hamiltonian introduced by Pirogov and Sinai in~\cite{PS1975} (see also Sinai~\cite{Sinai}, as well as, for example, Gross~\cite{Gross} for the case of more general potentials).  Let us also note that the notion of transition energy field is somehow related to the notion of relative energy, which proved itself useful in investigating the effect of transformations of equilibrium states (see, for example, Maes~\cite{Maes2,Maes1}).  However, unlike the Hamiltonian, the relative Hamiltonian and the relative energy, our objects are defined axiomatically by their natural and physically well-founded intrinsic properties.  Moreover, basing ourselves on the notion of transition energy field, we give a proper mathematical definition of the Hamiltonian without involving the notion of potential.

The argumentation used in this work follows the idea formulated by Sinai in~\cite{Sinai} that for infinite volume systems, the Hamiltonian cannot be defined as a function on the space of all configurations; however, it makes sense to consider the ``differences of the Hamiltonians'' for configurations which differ in a finite number of points.  At the same time, our considerations are in many ways related to Dobrushin's problem of description of specifications by means of systems of consistent one-point probability distributions with infinite boundary conditions, which was solved (by introducing the notion of 1--specification) in the works~\cite{DN1998, DN2001, DN2004} of the authors (see also Fern\'andez and Maillard~\cite{FM,FM2}).

The paper is organized as follows.  The necessary preliminaries are given in Section~\ref{Prelim}.  In Sections~\ref{TE} and~\ref{PE}, we introduce the notions of transition energy function and of (one-point) transition energy field, and show that the representation in Gibbsian form is a necessary and sufficient condition that a family of functions in finite volumes (resp.\ in lattice points) with infinite boundary conditions be a strictly positive specification (resp.\ 1--specification).  This result can, in our opinion, be considered as a justification of the Gibbs formula in the infinite volume case.  Moreover, its necessity part answers the problem formulated by D.\ Ruelle in~\cite{R2} of how wide the class of specifications, which can be represented in Gibbsian form, is.  Our answer is that any strictly positive specification admits Gibbsian representation (with some Hamiltonian satisfying our definition).  Finally, in Section~\ref{GenTheory}, the basic statements of the theory of Gibbs random fields are reformulated in terms of one-point transition energy fields.  The so-obtained description of random fields establishes a straightforward relationship between the probabilistic notion of (Gibbs) random field and the physical notion of (transition) energy, and so opens the possibility to directly apply probabilistic methods to mathematical problems of statistical physics.

\section{Preliminaries}
\label{Prelim}

Let $X$ be a non-empty finite set and $\mathbb{Z}^d$ be a $d$-dimensional integer lattice (a set of $d$-dimensional vectors with integer components), $d \geq 1$.  Without loss of generality, one can use any arbitrary countable set instead of $\mathbb {Z}^d$ in subsequent considerations, however we stick to tradition.

We assume that $X$ is endowed with the total $\sigma$-algebra $\cB$ and with the discrete topology $\cT$ $\bigl($that is, $\cB=\cT=\mathop{\mathrm{part}}(X)\bigr)$.

For any $S \subset \mathbb{Z}^d$, let $\W(S) = \{ V \subset S,\ \bil| V \bir| < \infty \}$ be the set of all finite subsets of~$S$.  For~$S = \mathbb{Z}^d$ we will use a simpler notation $\W$.  The elements of $\W$ are also often called \emph{finite volumes}.  Later on, the term \emph{lattice point} will be used not only for the elements $t \in \mathbb{Z}^d$, but also for the corresponding one-point volumes $\{t\}$.  The braces in the notation of such volumes will usually be omitted.

Let $S \subset \mathbb{Z}^d$, and let $f(\Lambda)$, $\Lambda\in \W(S)$, be some function on $\W(S)$.  In the sequel, the notation~$\lim \limits_{\Lambda \uparrow S} f(\Lambda) = a$, will mean that for any increasing sequence ${(\Lambda_n )}_{n \geq 1}$ of finite sets converging to~$S$~$\Bigl($that is, $\Lambda_n\in \W(S)$, ${\Lambda_n \subset \Lambda_{n+1}}$ and $\bigcup\limits_{n=1}^\infty \Lambda_n = S\Bigr)$, we have
\[
\lim \limits_{n \to \infty} f(\Lambda_n) = a.
\]

For any $S \subset \mathbb{Z}^d$, we denote $X^S = \bigl\{ ( x_t,\ t \in S ) \bigr\}$, $x_t \in X$, the space of \emph{configurations} on~$S$, that is, the set of all functions on $S$ taking values in $X$.  If $S = \Memptyset$, we assume $X^\Memptyset = \{ \Mbfemptyset \}$, where~$\Mbfemptyset$ is the \emph{empty configuration}.  For any $S, T \subset \mathbb{Z}^d$ such that $S \cap T = \Memptyset$ and any $x \in X^S$ and $y \in X^T$, we denote $x y$ the \emph{concatenation} of $x$ and $y$ defined as the configuration on $S \cup T$ equal to $x$ on~$S$ and to~$y$ on $T$.  For any $T \subset S\subset\mathbb{Z}^d$ and any ${x \in X^S}$, we denote $x_T$ the configuration $(x_t,\ t\in T)$ called \emph{restriction} of~$x$ on $T$.

A real-valued function $g$ on $X^S$, $S \subset \mathbb{Z}^d$, is called \emph{quasilocal} if
\[
\lim_{\Lambda \uparrow S} \sup_{x,y \in X^S \,:\, x_\Lambda = y_\Lambda} \bigl| g(x) - g(y) \bigr| = 0,
\]
or, equivalently, if $g$ is a uniform limit of \emph{local} functions (that is, functions depending only on values of configuration in a finite volume).  Note also that the quasilocality is nothing but continuity with respect to the topology $\cT^S$ and, taking into account that $X^S$ is compact, the strict positivity and the uniform nonnullness conditions are equivalent for quasilocal functions.

In the sequel, we adopt the following convention: if a finite subset of $\mathbb{Z}^d$ is used as a subscript in the notation of a function, then the function is defined on the space of configurations on this subset.  For $V\in\W$, a function $g_V$ (on the space $X^V$) will be also called \emph{function in (finite) volume}~$V$.  If a function in a finite volume $V\in\W$ has a superscript $\wb x \in X^{\mathbb{Z}^d \setminus V}$ in its notation, it will be called \emph{function in (finite) volume $V$ with (infinite) boundary condition $\wb x$}.  A fami\-ly~$\bigl\{ g_V^{\wb x},\ V \in \V,\, \wb x \in X^{\mathbb{Z}^d \setminus V} \bigr\}$ of functions, where $\V \subset \W$, will be called \emph{quasilocal} if for any $V \in \V$ and any $x\in X^V$, the function~$\wb x\mapsto g_V^{\wb x}(x)$ on $X^{\mathbb{Z}^d \setminus V}$ is quasilocal.

For $V \in \W$, a probability distribution on $X^V$ is given by a non-negative function~$\pb_V$ summing up to $1$.  In the case $V = \Memptyset$, there exist only one probability distribution defined by $\pb_\Memptyset (\Mbfemptyset) = 1$.  A probability distribution $\pb_V$, $V \in \W$, is called \emph{strictly positive} if $\pb_V (x) > 0$ for all $x \in X^V$.  For any $I\subset V \in \W$ and any probability distribution $\pb_V$, we denote ${(\pb_V)}_I$ the \emph{restriction} (\emph{marginal distribution}) of~$\pb_V$ on~$I$, defined by
\[
{(\pb_V)}_I (x) = \sum_{y \in X^{V \setminus I}} \pb_V(x y), \quad x \in X^I.
\]

Let $\cB^{\mathbb{Z}^d}$ be the $\sigma$-algebra generated by the cylinder (finite-dimensional) subsets of $X^{\mathbb{Z}^d }$.  A probability distribution $\PB$ on $\bigl( X^{\mathbb{Z}^d }, \cB^{\mathbb{Z}^d} \bigr)$ is called \emph{random field} (with \emph{state space} $X$).  The system of \emph{finite-dimensional distributions} of a random field $\PB$, will be denoted $\Pb = \{ \pb_V,\ V \in \W \}$.  Since the elements of $\Pb$ are the restrictions of the measure~$\PB$ on the corresponding $\sigma$-subalgebras of~$\cB^{\mathbb{Z}^d}$, they are consistent in the Kolmogorov sense: ${(\pb_V)}_I = \pb_I$ for all $I \subset V\in \W$.  The converse is also true (Kolmogorov's theorem): for any system of consistent probability distributions in finite volumes, there exists a unique random field having it as the system of finite-dimensional distributions.

In the present work we consider only \emph{strictly positive} random fields, that is, random fields having strictly positive finite-dimensional distributions.

A random field $\PB^{(p)}$ is called \emph{Bernoulli random field with parameter} $p\in(0,1)$ if for any~$V \in \W$, the finite-dimensional distribution of $\PB^{(p)}$ in the volume $V$ has the form
\[
\pb_V^{(p)} (x) = p^{\bil| x \bir|} (1-p)^{\bil| V \bir| - \bil| x \bir|}, \quad x \in X^V,
\]
where $\bil| x \bir| = \bigl| \{ t \in V \,:\, x_t = 1 \} \bigr|$ is the number of $1$'s in the configuration $x$.  This random field describes a system of independent random variables taking values $1$ and $0$ with probabilities $p$ and~$1-p$, respectively.

For a strictly positive random field $\PB$, its \emph{conditional distribution $\qb_V^{\widetilde x}$ in finite volume $V \in \W$ under finite boundary condition $\widetilde x \in X^\Lambda$}, $\Lambda \in \W(\mathbb{Z}^d\setminus V)$, is defined by
\[
\qb_V^{\widetilde x} (x) = \frac{\pb_{V \cup \Lambda} (x \widetilde x)}{\pb_\Lambda (\widetilde x)}, \quad x \in X^V.
\]

Let $\{ \partial t,\ t \in \mathbb{Z}^d \}$ be a system of neighborhoods on $\mathbb{Z}^d$, that is, a system of finite subsets $\partial t$ of the lattice $\mathbb{Z}^d$ such that $t \notin \partial t$, and $s \in \partial t$ if and only if $t \in \partial s$, $s \in \mathbb{Z}^d$.  A random field $\PB$ is called \emph{Markov random field} if its conditional distributions in finite volumes under finite boundary conditions satisfy the following property: for all $t \in \mathbb{Z}^d$, $\wb x \in X^{\mathbb{Z}^d \setminus t}$ and $V \in \W (\mathbb{Z}^d \setminus t)$ such that~$\partial t \subset V$, it holds $\qb_t^{\wb x_V} = \qb_t^{\wb x_{\partial t}}$.

For a random field $\PB$, the limits
\[
\qb_V^{\wb x} (x) = \lim_{\Lambda \uparrow \mathbb{Z}^d \setminus V} \qb_V^{\wb x_\Lambda} (x), \quad x \in X^V,
\]
exist for all $V \in \W$ and almost all (with respect to the measure $\PB$) configurations~$\wb x\in X^{\mathbb{Z}^d \setminus V}$.  Any family~$\Qb = \bigl\{ q_V^{\wb x},\ V \in \W,\, \wb x \in X^{\mathbb{Z}^d \setminus V} \bigr\}$ of probability distributions such that for any $V\in\W$, the equality $q_V^{\wb x} = \qb_V^{\wb x}$ holds for almost all (with respect to $\PB$) configurations $\wb x \in X^{\mathbb{Z}^d\setminus V}$, is called \emph{(version of) conditional distribution} of the random field~$\PB$.  In exactly the same way, any family~$\Qb^{(1)} = \bigl\{ q_t^{\wb x},\ t \in \mathbb{Z}^d,\, \wb x \in X^{\mathbb{Z}^d \setminus t} \bigr\}$ of probability distributions such that for any $t\in\mathbb{Z}^d$, the equality $q_t^{\wb x} = \qb_t^{\wb x}$ holds for almost all configurations $\wb x \in X^{\mathbb{Z}^d\setminus t}$, is called \emph{(version of) one-point conditional distribution} of the random field~$\PB$.

A family $\Qb = \bigl\{ q_V^{\wb x},\ V \in \W,\, \wb x \in X^{\mathbb{Z}^d \setminus V} \bigr\}$ of probability distributions in finite volumes with infinite boundary conditions is called \emph{specification} if its elements are consistent in the Dobrushin sense: for all $I \subset V \in \W$, $x \in X^{V \setminus I}$, $y \in X^I$ and $\wb x \in X^{\mathbb{Z}^d \setminus V}$, it holds
\begin{equation}
\label{DCC}
q_V^{\wb x} (x y) = {\bigl(q_V^{\wb x}\bigr)}_{V\setminus I} (x)\, q_I^{\wb x x} (y).
\end{equation}
These consistency conditions can be equivalently rewritten in the following form (see the paper~\cite{DN2004} of the authors): for all $I \subset V \in \W$, $x,u \in X^{V \setminus I}$, $y \in X^I$ and $\wb x \in X^{\mathbb{Z}^d \setminus V}$, it holds
\begin{equation}
\label{DCCeq}
q_V^{\wb x} (x y)\, q_{V\setminus I}^{\wb x y} (u) = q_V^{\wb x} (u y)\, q_{V\setminus I}^{\wb x y} (x).
\end{equation}

Let us note, that the elements of a conditional distribution of a random field $\PB$ satisfy the relations~\eqref{DCC} and~\eqref{DCCeq} for almost all (with respect to the measure $\PB$) configurations $\wb x \in X^{\mathbb{Z}^d \setminus V}$.  However, any random field has at least one version of conditional distribution being a specification (see Goldstein~\cite{Golds}, Preston~\cite{Prest} and Sokal~\cite{Sokal}).  Note also that if a strictly positive random field has a quasilocal version of conditional distribution, the latter is unique and is necessarily a specification (see the paper~\cite{DN2009} of the authors).  On the other hand, the quasilocality of a specification $\Qb$ guaranties the existence of a random field $\PB$ \emph{compatible with (the specification)}~$\Qb$, that is, having it as a version of conditional distribution (see, for example, Dobrushin~\cite{D1} or Georgii~\cite{Georg}).

In the present work we consider only \emph{strictly positive} specifications, that is, specification with strictly positive elements.

The main example of specification is the Gibbsian specification with a given potential. Let us give its definition.

A family $\bPhi = \bigl\{ \Phi_J,\ J \in \W \setminus \{\Memptyset\} \bigr\}$ of functions is called \emph{interaction potential} (or simply \emph{potential}) if for every $V \in \W$ and $\wb x \in X^{\mathbb{Z}^d \setminus V}$, there exist finite limits
\begin{equation}
\label{Phi_lim}
\hh_V^{\wb x} (x) = \lim_{\Lambda \uparrow \mathbb{Z}^d \setminus V} \sum_{\Memptyset \neq I \subset V} \sum_{ J \subset \Lambda} \Phi_{I\cup J} ( x_I \wb x_J ), \quad x \in X^V.
\end{equation}
The potential $\bPhi$ is said to be \emph{uniformly convergent} if the convergence in~\eqref{Phi_lim} is uniform with respect to~$\wb x$.  The family $\HH_\bPhi = \bigl\{ \hh_V^{\wb x},\ V \in \W,\, \wb x \in X^{\mathbb{Z}^d \setminus V} \bigr\}$ is called \emph{Hamiltonian corresponding to the potential~$\bPhi$}.

The family $\Qb_\bPhi=\bigl\{ q_V^{\wb x},\ V \in \W,\, \wb x \in X^{\mathbb{Z}^d \setminus V} \bigr\}$ of probability distributions
\[
q_V^{\wb x} (x) = \frac{\exp \{ - \hh_V^{\wb x} (x) \}}{\sum\limits_{z \in X^V} {\exp \{ - \hh_V^{\wb x} (z) \}} }, \quad x \in X^V,
\]
(defined by the Gibbs formula) is a strictly positive specification.  Indeed, its elements are strictly positive, and it is not difficult to verify that they satisfy the consistency conditions~\eqref{DCCeq}.  The specification $\Qb_\bPhi$ is called \emph{Gibbsian specification with potential~$\bPhi$}.

According to Dobrushin, a random field $\PB$ is called \emph{Gibbs random field with potential $\bPhi$} if it is compatible with the specification $\Qb_\bPhi$.  Note that if the potential $\bPhi$ is uniformly convergent, the Hamiltonian $\HH_\bPhi$ is quasilocal, which in turn implies the quasilocality of the specification $\Qb_\bPhi$ and thereby guarantees the existence of a Gibbs random field with potential~$\bPhi$.

A family $\Qb^{(1)} = \bigl\{ q_t^{\wb x},\ t \in \mathbb{Z}^d,\, \wb x \in X^{\mathbb{Z}^d \setminus t} \bigr\}$ of strictly positive probability distributions in lattice points with infinite boundary conditions is called \emph{strictly positive 1--specification} if its elements are consistent in the following sense: for all $t,s \in \mathbb{Z}^d$, $x,u \in X^{t}$, $y,v \in X^{s}$ and $\wb x \in X^{\mathbb{Z}^d \setminus \{ t,s\}}$, it holds
\begin{equation}
\label{Sogl1}
q_t^{\wb x y} (x)\,q_s^{\wb x x} (v)\,q_t^{\wb x v} (u)\,q_s^{\wb x u} (y) = q_s^{\wb x x} (y)\,q_t^{\wb x y} (u)\,q_s^{\wb x u} (v)\,q_t^{\wb x v} (x).
\end{equation}
It should be noted that this definition is a particular case of a more general (not requiring the strict positivity assumption) definition of 1--specification introduced in the paper~\cite{DN2004} of the authors.

The main example of 1--specification is the following.  Let $\bPhi$ be an interaction potential.  The family $\Qb_\bPhi^{(1)}=\bigl\{ q_t^{\wb x},\ t \in\mathbb{Z}^d,\, \wb x \in X^{\mathbb{Z}^d \setminus t} \bigr\}$ of one-point probability distributions
\[
q_t^{\wb x} (x) = \frac{\exp \{ - \hh_t^{\wb x} (x) \}}{\sum\limits_{z \in X^{t}} {\exp \{ - \hh_t^{\wb x} (z) \}} }, \quad x \in X^{t},
\]
where
\[
\hh_t^{\wb x} (x) = \lim_{\Lambda \uparrow \mathbb{Z}^d \setminus t}\;\sum_{ J \subset \Lambda} \Phi_{t\cup J } ( x \wb x_J ), \quad x \in X^{t},
\]
is a strictly positive 1--specification.  Indeed, its elements are strictly positive, and it is easy to see that they satisfy the consistency conditions~\eqref{Sogl1}.  The family $\HH_\bPhi^{(1)} = \bigl\{ \hh_t^{\wb x},\ t \in\mathbb{Z}^d,\, \wb x \in X^{\mathbb{Z}^d \setminus t} \bigr\}$ is called \emph{one-point Hamiltonian corresponding to the potential $\bPhi$}, while the 1--specification $\Qb_\bPhi^{(1)}$ is called \emph{Gibbsian 1--specification with potential~$\bPhi$}.  Note that if the potential $\bPhi$ is uniformly convergent, both the one-point Hamiltonian~$\HH_\bPhi^{(1)}$ and the 1--specification $\Qb_\bPhi^{(1)}$ are quasilocal.

Dobrushin's problem of description of specifications by means of systems of consistent one-point probability distributions with infinite boundary conditions~(see~\cite{D1, D4}) was solved by the authors in~\cite{DN1998,DN2001} under the weak positivity condition (as well as under the strict positivity condition) and in~\cite{DN2004} under the newly-introduced very weak positivity condition. The case of a general (not necessarily finite) state space was considered by Fern\'andez and Maillard in~\cite{FM} under an alternative nonnullness condition and in~\cite{FM2} under the extension to this case of the very weak positivity condition.

In the strictly positive case, the solution of Dobrushin's problem is given by the following theorem (see the works~\cite{DN1998,DN2001,DN2004} of the authors).

\begin{Thm}
\label{DN}
A family $\Qb^{(1)} = \bigl\{ q_t^{\wb x},\ t \in \mathbb{Z}^d,\, \wb x \in X^{\mathbb{Z}^d \setminus t} \bigr\}$ of functions in lattice points with infinite boundary conditions will be a subsystem of some strictly positive specifications $\Qb$ if and only if~$\Qb^{(1)}$ is a strictly positive 1--specification.  In this case, the specification $\Qb$ is uniquely determined.
\end{Thm}

Note that as a matter of fact, an explicit formula expressing the elements of the specification~$\Qb$ by the elements of the 1--specification $\Qb^{(1)}$ was obtained in the proof of this theorem (see also the work~\cite{DN2009} of the authors).  This formula in particular implies that the specification $\Qb$ is quasilocal if and only if the 1--specification $\Qb^{(1)}$ is quasilocal, and that the set of random fields compatible with~$\Qb$ coincides with the set of random fields \emph{compatible with (the 1--specification)} $\Qb^{(1)}$, that is, having $\Qb^{(1)}$ as a version of one-point conditional distribution.  Thereby, it becomes possible to reformulate Dobrushin's theory (of descriptions of random fields by means of specifications) in terms of 1--specifications.

\section{Transition energy}
\label{TE}

First of all, let us note that the potential energy (Hamiltonian) of a physical system is defined up to an additive constant and therefore cannot be directly measured. So, only the change of the potential energy has a physical meaning, and not the potential energy itself.

The main idea of the approach we propose to justify the Gibbs formula is to use, instead of the Hamiltonian, the function giving the energy cost of transition of the physical system from one state to another.  Unlike the Hamiltonian, the transition energy function is uniquely determined (for a given physical system) and, as it will be shown latter, admits a description by its intrinsic properties that have a clear physical meaning.  Besides, in many cases the use of transition energy greatly simplifies the considerations.

\subsection{Finite volume systems}

Let $V \in \W$ be some finite volume.  Consider a physical system in the volume $V$.  For each pair of states $x, u \in X^V$, denote $\te_V (x,u)$ the amount of energy needed to move the system from the state~$x$ to the state~$u$.  Clearly,~$\te_V$ satisfies the relation
\begin{equation}
\label{Delta_sum}
\te_V (x,u) = \te_V (x,z) + \te_V (z,u), \quad x, u, z \in X^V,
\end{equation}
which is nothing but the energy conservation law.  These physical considerations leads us to the following definition.

A function~$\te_V$ of two variables (defined on the set of all pairs of configurations from $X^V$) will be called \emph{transition energy in volume $V$} if it satisfies the relation~\eqref{Delta_sum}.  Note that in the case $V = \Memptyset$, there is only one transition energy in volume $\Memptyset$ defined by $\te_\Memptyset(\Mbfemptyset,\Mbfemptyset) = 0$.  Note also that the relation~\eqref{Delta_sum} implies
\[
\te_V (x,x) = 0 \quad \text{and} \quad \te_V (x,u) = - \te_V (u,x), \quad x, u \in X^V.
\]

The following simple proposition provides a full description of strictly positive probability distributions in volume $V$ in terms of transition energies in volume $V$ and will also play an important role in the study of infinite volume systems.

\begin{Prop}
\label{GibbsFormProb}
Let $V\in\W$.  A function $\pb_V$ will be a strictly positive probability distribution if and only if it has the Gibbsian form
\begin{equation}
\label{GFP1}
\pb_V(x) = \frac{ \exp \{\te_V(x,u)\} }{\sum \limits_{z \in X^V} \exp \{\te_V(z,u)\} }, \quad x \in X^V,
\end{equation}
where $\te_V$ is some transition energy in volume $V$, and the configuration $u\in X^V$ is arbitrary.  In this case, $\te_V$ is uniquely determined by $\pb_V$ according to the formula
\begin{equation}
\label{GFP2}
\te_V (x,u) = \ln \frac{\pb_V (x)}{\pb_V (u)}, \quad x,u \in X^V.
\end{equation}
\end{Prop}

\begin{proof}
Let $\pb_V$ be a strictly positive probability distribution.  Clearly, the function $\te_V$ defined by the formula~\eqref{GFP2} satisfies the relation~\eqref{Delta_sum} and, therefore, is a transition energy in volume $V$.  Further, for any $u\in X^V$, we can write
\[
\pb_V(x) = \frac{ \pb_V(x) }{\sum \limits_{z \in X^V} \pb_V(z) } = \frac{ \pb_V(x)/\pb_V(u) }{\sum \limits_{z \in X^V} \pb_V(z)/\pb_V(u) } = \frac{ \exp \{\te_V(x,u)\} }{\sum \limits_{z \in X^V} \exp \{\te_V(z,u)\} },\quad x\in X^V,
\]
and so $\pb_V$ has the required Gibbsian form~\eqref{GFP1}.

Let now $\te_V$ be a transition energy in volume $V$, and let us define $\pb_V$ by the formula~\eqref{GFP1}.  First of all, let us show the correctness of this definition, that is, check that $\pb_V$ does not depend on the choice of~$u$.  Indeed, by virtue of the relation~\eqref{Delta_sum}, for any $w \in X^V$, we have
\[
\frac{\exp \{\te_V(x,u)\}}{\sum \limits_{z \in X^V} {\exp \{\te_V(z,u)\}}} = \frac{\exp \{\te_V(x,w) + \te_V(w,u)\}}{\sum \limits_{z \in X^V} {\exp \{\te_V(z,w) + \te_V(w,u)\} }} = \frac{\exp \{\te_V(x,w)\}}{\sum \limits_{z \in X^V} {\exp \{\te_V(z,w)\}}}.
\]
Also, it is evident that $\pb_V$ is a strictly positive probability distribution.

Finally, if $\pb_V$ have the Gibbian form~\eqref{GFP1}, we can write
\[
\frac{\pb_V(x)}{\pb_V(u)} = \frac{\exp\{\te_V(x,u)\}}{\exp\{\te_V(u,u)\}} = \exp\{\te_V(x,u)\},\quad x,u \in X^V,
\]
and hence $\te_V$ is uniquely determined by $\pb_V$ according to the formula~\eqref{GFP2}.
\end{proof}

Let us note that the general solution of the functional equation~\eqref{Delta_sum} has the form
\[
\te_V (x,u) = \hh_V (u) - \hh_V (x),
\]
where $\hh_V$ is an arbitrary function which, for a physical system corresponding to $\te_V$, can be interpreted as potential energy (Hamiltonian).  Thus, in the case of finite volume systems, the Hamiltonian is not required to posses any intrinsic properties, and the notion of the Hamiltonian does not need a special definition.

Note also that $\hh_V$ is not uniquely determined by $\te_V$.  It is rather determined up to an arbitrary additive constant (corresponding to the choice of the zero level of potential energy).  However, for the probability distribution $\pb_V$ corresponding to $\te_V$, we can write
\[
\pb_V (x) = \frac{ \exp \{ \te_V (x,u) \} }{ \sum\limits_{z \in X^V } { \exp \{ \te_V (z,u) \} }} = \frac{ \exp \{ - \hh_V (x) \} }{\sum\limits_{z \in {X^V}} {\exp \{ - \hh_V (z) \}} }, \quad x \in X^V.
\]
These relations yield, in particular, the correctness of the Gibbs formula.  Namely, although $\hh_V$ is not uniquely determined for a given physical system (and hence a given $\te_V$), the corresponding Gibbs distribution does not depend on the choice of $H_V$.  Moreover, according to Proposition~\ref{GibbsFormProb}, the representation in Gibbsian form is a necessary and sufficient condition that a function be a probability distribution.  This fact can, in our opinion, be considered as a justification of the Gibbs formula (in the finite volume case) on the base of the energy conservation law.

\begin{remark}
Let us note that the Gibbs formula has a number of justifications based on different physical considerations as, for example, the free energy minimum principle or the maximum entropy principle (see, for example, Pfister~\cite{Pfister}).  Let us bring one of them which, in our opinion, is the simplest.  If we assume that, up to a normalizing constant, the probability is a strictly positive, continuous and strictly decreasing function of the energy $\bigl($that is, $\pb_V(x)=C f\bigl(\hh_V(x)\bigr)$, $x\in X^V\bigr)$, and that the choice of the zero level of the potential energy (addition of a constant to the function $\hh_V$) does not change the probability distribution~$\pb_V$, it is not difficult to show that $f(t)=\alpha\exp\{-\beta t\}$ with some strictly positive parameters $\alpha$ and $\beta$.  As the constant~$\alpha$ cancels out after the normalization $\Bigl($since~$C^{-1}=\sum\limits_{z \in {X^V}}f\bigl(\hh_V(z)\bigr)=\alpha\sum\limits_{z \in {X^V}}\exp\bigl(-\beta\hh_V(z)\bigr)\Bigr)$, without loss of generality one can take~$\alpha=1$.  Note also that since the probability should not depend on the choice of the energy units, the presence of the multiplicative constant $\beta$ before the energy~$\hh_V$ is natural (roughly speaking, the multiplication by $\beta$ makes the energy dimensionless).  In the present work, we always include the constant $\beta$ in the energy~$\hh_V$ (that is, directly consider dimensionless energies).\qed
\end{remark}

\subsection{Infinite volume systems}

Now, let us turn to the systems whose states are configurations on $\mathbb{Z}^d$.  Note that the transition energy for such systems cannot be defined on all pairs of configurations, because on pairs of configurations that differ in an infinite number of points, it will, as a rule, take infinite value.  In this connection, already Sinai noted in his book~\cite{Sinai} that for infinite volume systems, the Hamiltonian cannot be defined as a function on $X^{\mathbb{Z}^d}$; however, it makes sense to consider the ``differences of the Hamiltonians'' for configurations $\wb x, \wb u \in X^{\mathbb{Z}^d}$ that differ in a finite number of points.  This is the way we will follow.

Let $T$ be the set of all pairs of configurations on $\mathbb{Z}^d$ that differ in at most a finite number of points.  For $V \in \W$ and $\wb x \in X^{\mathbb{Z}^d \setminus V}$, we denote $T_V^{\wb x}$ the set of all pairs of configurations on $\mathbb{Z}^d$ coinciding with $\wb x$ on $\mathbb{Z}^d \setminus V$.  Clearly, $T = \bigcup \limits_{V \in \W\vphantom{\int^T}}\:\bigcup \limits_{\smash{\wb x \in X^{\mathbb{Z}^d \setminus V}}\vphantom{\int^T}} T_V^{\wb x}$.

Any function $\TE$ on $T$ satisfying the relation
\begin{equation}
\label{Delta0}
\TE (\wb u, \wb w) = \TE (\wb u, \wb z) + \TE (\wb z, \wb w)
\end{equation}
for all $(\wb u, \wb z), (\wb z, \wb w) \in T$ will be called \emph{transition energy function} (or simply \emph{transition energy}).

A basic example of transition energy function is the following.  Let  $\bPhi = \bigl\{ \Phi_J,\ J \in \W \setminus \{\Memptyset\} \bigr\}$ be some finite-range interaction potential (that is, $\Phi_J=0$ as soon as $\diam(J)>R$ for some $R\geq 0$), and put
\[
\TE_\bPhi(\wb u, \wb w)=\sum_{J \in \W \setminus \{\Memptyset\}}\bigl(\Phi_J(\wb w_J)-\Phi_J(\wb u_J)\bigr),\quad (\wb u,\wb w)\in T.
\]
The function $\TE_\bPhi$ on $T$ is well defined (since the above sum contains a finite number of terms) and clearly satisfies the relation (8).  Let us note, that (up to a change of sign) $\TE_\bPhi$ is the \emph{relative Hamiltonian corresponding to the potential~$\bPhi$}, introduced by Pirogov and Sinai in~\cite{PS1975} (see also Sinai~\cite{Sinai}, as well as, for example, Gross~\cite{Gross} for the case of more general potentials).  However, we would like to stress here, that in contrary to the relative Hamiltonian (defined in terms of some given potential), the transition energy function is defined axiomatically by its intrinsic properties.

\subsubsection{Transition energy field}

A transition energy $\TE$ defines a family $\Te = \bigl\{ \te_V^{\wb x},\ V \in \W,\, \wb x \in X^{\mathbb{Z}^d \setminus V} \bigr\}$ of functions of two variables in finite volumes with infinite boundary conditions by
\begin{equation}
\label{Delta_V_def}
\te_V^{\wb x} (x, u) = \TE (x \wb x, u \wb x), \quad x,u \in X^V.
\end{equation}
Since
\[
\te_V^{\wb x} (x,u) = \TE (x \wb x, u \wb x) = \TE (x \wb x, z \wb x) + \TE (z \wb x, u \wb x) = \te_V^{\wb x} (x,z) + \te_V^{\wb x} (z,u), \quad x,u,z \in X^V,
\]
the elements of $\Te$ are transition energies in finite volumes (with infinite boundary conditions).  Besides, they are consistent in the following sense: for all $V, I \in \W$ such that~$V \cap I = \Memptyset$ and all~$x,u \in X^V$, $y \in X^I$ and ${\wb x \in X^{\mathbb{Z}^d \setminus (V \cup I)}}$, it holds
\begin{equation}
\label{Delta2}
\te_{V \cup I}^{\wb x} (x y, u y) = \te_V^{\wb x y} (x,u).
\end{equation}
To see this, it is sufficient to note that
\[
\te_{V \cup I}^{\wb x} (x y, u y) = \TE (x y \wb x, u y \wb x) \quad \text{and} \quad \te_V^{\wb x y} (x,u) = \TE (x y \wb x, u y \wb x).
\]

A family $\Te = \bigl\{ \te_V^{\wb x},\ V \in \W,\, \wb x \in X^{\mathbb{Z}^d \setminus V} \bigr\}$ of transition energies in finite volumes with infinite boundary conditions will be called \emph{transition energy field} if its elements satisfy the consistency conditions~\eqref{Delta2}.

Let us note that the relations \eqref{Delta_V_def} and \eqref{Delta2} are very natural from the physical point of view. Indeed, when a system moves from one configuration (state) to another, the energy is spent only on changing the state of the system in the volume in which these configurations are different, while the state of the system in the volume in which they coincide plays the role of a boundary condition.

It follows from the above considerations that any transition energy $\TE$ generates a transition energy field $\Te$.  The converse is also true.

\begin{Prop}
Let $\Te = \bigl\{ \te_V^{\wb x},\ V \in \W,\, \wb x \in X^{\mathbb{Z}^d \setminus V} \bigr\}$ be a transition energy field.  Then there exists a unique transition energy $\TE$ generating the transition energy field $\Te$.
\end{Prop}

\begin{proof}
For each pair of configurations $(x\wb x, u\wb x) \in T_V^{\wb x}$, where $V \in \W$, $x,u \in X^V$ and $\wb x \in X^{\mathbb{Z}^d \setminus V}$, we put
\[
\TE(x \wb x, u \wb x) = \te_V^{\wb x} (x,u).
\]
This defines the function $\TE$ on the whole $T$, while the equation~\eqref{Delta2} guarantees the correctness of the definition.

To show that $\TE$ is a transition energy, let us verify~\eqref{Delta0}.  Let $(\wb u, \wb z)$, $(\wb z, \wb w)\in T$.  Then there exist $V \in \W$ and $\wb x \in X^{\mathbb{Z}^d \setminus V}$ such that $(\wb u, \wb z), (\wb z, \wb w) \in T_V^{\wb x}$, and hence $\wb u = u \wb x$, $\wb z = z \wb x$ and $\wb w = w \wb x$, where $u = \wb u_V$, $z = \wb z_V$, $w = \wb w_V$ and $\wb x = \wb u_{\mathbb{Z}^d \setminus V} = \wb z_{\mathbb{Z}^d \setminus V} = \wb w_{\mathbb{Z}^d \setminus V}$.  Taking into account that $\te_V^{\wb x}$ is a transition energy in volume $V$ $\bigl($and hence satisfy the relation~\eqref{Delta_sum}$\bigr)$, we get
\begin{align*}
\TE (\wb u, \wb w) &= \TE ( u \wb x, w \wb x) = \te_V^{\wb x} (u,w) = \te_V^{\wb x} (u,z) + \te_V^{\wb x} (z,w) \\[7pt]
&= \TE ( u \wb x, z \wb x) + \TE ( z \wb x, w \wb x) = \TE (\wb u, \wb z) + \TE (\wb z, \wb w).
\end{align*}

Finally, the uniqueness of $\TE$ is trivial.
\end{proof}

So, specifying a transition energy $\TE$ on $T$ is equivalent to specifying a transition energy field~$\Te$, and in the subsequent considerations we will exclusively use  the latter.

Note that the conditions~\eqref{Delta2} of consistency of the elements of a transition energy field are equivalent to the following ones: for all $V,I \in \W$ such that $V \cap I = \Memptyset$ and all $x,u \in X^V$, $y,v \in X^I$ and $\wb x\in X^{\mathbb{Z}^d\setminus (V\cup I)}$, it holds
\begin{equation}
\label{Delta2.2}
\te_{V \cup I}^{\wb x} (x y, u v) = \te_V^{\wb x y} (x,u) + \te_{I}^{\wb x u} (y,v).
\end{equation}
Indeed, let~\eqref{Delta2} be fulfilled.  Since $\te_{V \cup I}^{\wb x}$ is a transition energy in volume $V \cup I$, we can write
\[
\te_{V \cup I}^{\wb x} (x y, u v) = \te_{V \cup I}^{\wb x} (x y, u y) + \te_{V \cup I}^{\wb x} (u y, u v) = \te_V^{\wb x y} (x,u) + \te_{I}^{\wb x u} (y,v).
\]
Conversely, putting $v=y$ in~\eqref{Delta2.2} and taking into account that $\te_{I}^{\wb x u} (y,y) = 0$ (since $\te_{I}^{\wb x u}$ is a transition energy in volume $I$), we obtain~\eqref{Delta2}.

Note also that the relation~\eqref{Delta2.2} has the following (natural) physical meaning: the energy $\te_{V \cup I}^{\wb x} (x y, u v)$ needed to change the state of a system from $x y$ to $u v$ in the volume $V\cup I$ under the boundary condition $\wb x$ is equal to the sum of energies $\te_V^{\wb x y} (x,u)$ and $\te_{I}^{\wb x u} (y,v)$ needed to first change the state of the system from $x$ to $u$ in the volume $V$ under the boundary condition $\wb x y$, and then from $y$ to $v$ in the volume $I$ already under the boundary condition $\wb x u$.

The theorem given below establishes a relationship between a strictly positive specification $\Qb$ and a transition energy field $\Te$.

\begin{Thm}
\label{GibbsFormSpec}
A family $\Qb = \bigl\{ q_V^{\wb x},\ V \in \W,\, \wb x \in X^{\mathbb{Z}^d \setminus V} \bigr\}$ of functions in finite volumes with infinite boundary conditions will be a strictly positive specification if and only if its elements have the Gibbsian form
\[
q_V^{\wb x} (x) = \frac{ \exp \{\te_V^{\wb x} (x,u)\} }{\sum \limits_{z \in X^V} {\exp \{\te_V^{\wb x} (z,u)\} }}, \quad x \in X^V,
\]
where $\Te=\bigl\{ \te_V^{\wb x},\ V \in \W,\, \wb x \in X^{\mathbb{Z}^d \setminus V} \bigr\}$ is some transition energy field, and the configuration $u\in X^V$ is arbitrary.  In this case, the transition energy field $\Te$ is uniquely determined.
\end{Thm}

\begin{proof}
Let $\Qb = \bigl\{ q_V^{\wb x},\ V \in \W,\, \wb x \in X^{\mathbb{Z}^d \setminus V} \bigr\}$ be a strictly positive specification.  According to Proposition~\ref{GibbsFormProb}, for each $V \in \W$ and~$\wb x \in X^{\mathbb{Z}^d \setminus V}$, the strictly positive probability distribution $q_V^{\wb x}$ has the required Gibbsian form with
\[
\te_V^{\wb x} (x,u) = \ln \frac{q_V^{\wb x}(x)}{q_V^{\wb x}(u)}, \quad x,u \in X^V,
\]
and $\te_V^{\wb x}$ is a (unique) transition energy in volume $V$.  To conclude the proof of the necessity, it remains to show that the family $\Te=\bigl\{ \te_V^{\wb x},\ V \in \W,\, \wb x \in X^{\mathbb{Z}^d \setminus V} \bigr\}$ is a transition energy field, that is, to check that its elements satisfy the consistency conditions~\eqref{Delta2}.

By virtue of the conditions~\eqref{DCCeq} of consistency of the elements of~$\Qb$, for all $V,I \in \W$ such that~$V \cap I = \Memptyset$ and all~$x,u \in X^V$, $y \in X^I$ and $\wb x\in X^{\mathbb{Z}^d\setminus (V\cup I)}$, we have
\[
\frac{ q_{V \cup I}^{\wb x}(x y) }{ q_{V \cup I}^{\wb x} (u y) } = \frac{ q_V^{\wb x y} (x) }{ q_V^{\wb x y} (u) },
\]
which yields
\[
\te_{V \cup I}^{\wb x} (x y, u y) = \ln \frac{q_{V \cup I}^{\wb x} (x y)}{q_{V \cup I}^{\wb x} (u y)} = \ln \frac{q_V^{\wb x y} (x)}{q_V^{\wb x y} (u)} = \te_V^{\wb x y} (x,u).
\]

Let now $\Te=\bigl\{ \te_V^{\wb x},\ V \in \W,\, \wb x \in X^{\mathbb{Z}^d \setminus V} \bigr\}$ be a transition energy field.  According to Proposition~\ref{GibbsFormProb}, for each $V \in \W$ and~$\wb x \in X^{\mathbb{Z}^d \setminus V}$, the function
\[
q_V^{\wb x} (x) = \frac{\exp \{\te_V^{\wb x} (x,u)\}}{\sum \limits_{z \in X^V} {\exp \{\te_V^{\wb x} (z,u)\} }}, \quad x \in X^V,
\]
is a strictly positive probability distribution (not depending on $u \in X^V$).  It remains to show that the elements of the family $\Qb=\bigl\{ q_V^{\wb x},\ V \in \W,\, \wb x \in X^{\mathbb{Z}^d \setminus V} \bigr\}$ are consistent in the Dobrushin sense or, equivalently, that for all $V,I \in \W$ such that $V \cap I = \Memptyset$ and all~$x,u \in X^V$, $y \in X^I$ and~$\wb x\in X^{\mathbb{Z}^d\setminus (V\cup I)}$, it holds
\[
\frac{ q_{V \cup I}^{\wb x}(x y) }{ q_{V \cup I}^{\wb x} (u y) } = \frac{ q_V^{\wb x y} (x) }{ q_V^{\wb x y} (u) }.
\]

The validity of the last relation follows from the chain of equalities
\[
\ln \frac{q_{V \cup I}^{\wb x} (x y)}{q_{V \cup I}^{\wb x} (u y)} = \te_{V \cup I}^{\wb x} (x y, u y) = \te_V^{\wb x y} (x,u) = \ln \frac{q_V^{\wb x y} (x)}{q_V^{\wb x y} (u)},
\]
which is true by virtue of the conditions~\eqref{Delta2} of consistency of the elements of $\Te$.
\end{proof}

It follows from this theorem that any strictly positive specification admits a representation in terms of a transition energy field.  In addition, it is not difficult to see that the specification will be quasilocal if and only if the corresponding transition energy field is quasilocal (for all $V \in \W$ and $x,u\in X^V$, the function $\wb x\mapsto \te_V^{\wb x}(x,u)$ on $X^{\mathbb{Z}^d \setminus V}$ is quasilocal).

We call a specification \emph{Gibbsian}, if it is strictly positive and quasilocal (see, for example, Criterion 1.1 in the work~\cite{DN2009} of the authors).  Thus, a strictly positive specification will be Gibbsian if and only if the corresponding transition energy field is quasilocal.

\subsubsection{One-point transition energy field}

The mentioned in Section~\ref{Prelim} solution of Dobrushin's problem provides a description of random fields by means of 1--specifications, which is much more economical and efficient than the usual description based on specifications.  The aim of the present section is to establish a relationship between strictly positive 1--specifications and appropriately consistent systems of transition energies in lattice points with infinite boundary conditions, similar to the relationship between strictly positive specifications and transition energy fields.

Consider the subsystem $\Te^{(1)} = \bigl\{ \te_t^{\wb x},\ t \in \mathbb{Z}^d,\, \wb x \in X^{\mathbb{Z}^d \setminus t} \bigr\}$ of some transition energy field $\Te$.  The elements of $\Te^{(1)}$ are transition energies in lattice points with infinite boundary conditions.  Moreover, they are consistent in the following sense: for all $t,s \in \mathbb{Z}^d$, $x,u \in X^{t}$, $y,v \in X^{s}$ and~$\wb x \in X^{\mathbb{Z}^d \setminus \{ t,s\}}$, it holds
\begin{equation}
\label{DeltaT2}
\te_t^{\wb x y} (x,u) + \te_s^{\wb x u} (y,v) = \te_s^{\wb x x} (y,v) + \te_t^{\wb x v} (x,u).
\end{equation}
Indeed, according to the consistency conditions~\eqref{Delta2.2} of the elements of~$\Te$, both sides of the relation~\eqref{DeltaT2} are equal to~$\te_{\{t,s\}}^{\wb x} (x y, u v)$.

A family $\Te^{(1)} = \bigl\{ \te_t^{\wb x},\ t \in \mathbb{Z}^d,\, \wb x \in X^{\mathbb{Z}^d \setminus t} \bigr\}$ of transition energies in lattice points with infinite boundary conditions will be called \emph{one-point transition energy field} if its elements satisfy the consistency conditions~\eqref{DeltaT2}.

Let us note that the relation~\eqref{DeltaT2} has the following (natural) physical meaning.  Suppose it is necessary to move a physical system from the state $x y\wb x$ to the state~$u v\wb x$.  This can be done in two ways: either by first changing the state of the system from $x$ to $u$ in the point $t$ under the boundary condition~$\wb x y$, and then from $y$ to $v$ in the point $s$ already under the boundary condition~$\wb x u$; or by first changing the state of the system from $y$ to $v$ in the point $s$ under the boundary condition~$\wb x x$, and then from $x$ to $u$ in the point~$t$ already under the boundary condition~$\wb x v$.  The amount of energy spent in the first case will be equal to $\te_t^{\wb x y} (x,u) + \te_s^{\wb x u} (y,v)$, while in the second case it will be equal to $\te_s^{\wb x x} (y,v) + \te_t^{\wb x v} (x,u)$.  Naturally, the same amount of energy must be spent in both cases.

\begin{Thm}
\label{Delta1-Delta}
A family $\Te^{(1)} = \bigl\{ \te_t^{\wb x},\ t \in \mathbb{Z}^d,\, \wb x \in X^{\mathbb{Z}^d \setminus t} \bigr\}$ of functions of two variables in lattice points with infinite boundary conditions will be a subsystem of some transition energy field $\Te$ if and only if $\Te^{(1)}$ is a one-point transition energy field.  In this case, the transition energy field $\Te$ is uniquely determined.
\end{Thm}

\begin{proof}
The necessity was already shown in the beginning of this section.  We therefore turn to the proof of the sufficiency.

Let $\Te^{(1)} = \bigl\{ \te_t^{\wb x},\ t \in \mathbb{Z}^d,\, \wb x \in X^{\mathbb{Z}^d \setminus t} \bigr\}$ be a transition energy field.  For all $V \in \W$, $x, u \in X^V$ and ${\wb x \in X^{\mathbb{Z}^d \setminus V}}$, we put
\begin{equation}
\label{Delta_V-Delta_t}
\te_V^{\wb x} (x,u) = \te_{t_1}^{\wb x x_2 x_3\cdots x_n } ( x_1, u_1 )
+ \te_{t_2}^{\wb x u_1 x_3 \cdots x_n} ( x_2, u_2 ) + \cdots + \te_{t_n}^{\wb x u_1 u_2 \cdots u_{n - 1} } ( x_n, u_n ),
\end{equation}
where $t_1, t_2,\ldots,t_n$ is some enumeration of the elements of the set $V$, and $x_i = x_{t_i}$, $u_i = u_{t_i}$ for \hbox{$i = 1,2,\ldots,n$} (in the case~$V=\Memptyset$, we use the convention that the sum of zero elements is zero).  Let us show that~$\te_V^{\wb x} (x,u)$ does not depend on the order of enumeration of the elements of~$V$.  Indeed, since any permutation of the elements $t_1 ,t_2 ,\ldots,t_n $ can be decomposed into a finite number of transpositions of adjacent elements $t_{k-1}$ and $t_k$, $k = 2,\ldots,n$, it is sufficient to show that
\begin{multline*}
\te_{t_{k-1}}^{\wb x u_1 \cdots u_{k-2} x_k \cdots x_n} ( x_{k-1}, u_{k-1} ) + \te_{t_k}^{\wb x u_1 \cdots u_{k-1} x_{k+1} \cdots x_n} ( x_k, u_k ) \\[7pt]
= \te_{t_k}^{\wb x u_1 \cdots u_{k-2} x_{k-1} x_{k+1} \cdots x_n} ( x_k, u_k ) + \te_{t_{k-1}}^{\wb x u_1 \cdots u_{k-2} u_k x_{k+1} \cdots x_n} ( x_{k-1}, u_{k-1} ),
\end{multline*}
which is guaranteed by the consistency conditions~\eqref{DeltaT2}.

Further, let $V,I\in \W$ be such that $V\cap I=\Memptyset$, and let $t_1, t_2,\ldots,t_n$ and $s_1, s_2,\ldots,s_m$ be some enumerations of the elements of the sets $V$ and $I$, respectively.  For any $x,u \in X^V$, $y,v \in X^I$ and~$\wb x \in X^{\mathbb{Z}^d \setminus (V \cup I)}$, we can write
\begin{align*}
\te_{V \cup I}^{\wb x} (x y, u v) &= \te_{t_1}^{\wb x x_2 x_3\cdots x_n y } ( x_1, u_1 ) + \te _{t_2}^{\wb x u_1 x_3 \cdots x_n y} ( x_2, u_2 ) + \cdots + \te_{t_n}^{\wb x u_1 u_2 \cdots u_{n - 1} y} ( x_n, u_n ) \\[7pt]
&\qquad\qquad + \te_{s_1}^{\wb x u y_2 y_3 \cdots y_m} ( y_1, v_1 ) + \te_{s_2}^{\wb x u v_1 y_3 \cdots y_m} ( y_2, v_2 ) + \cdots + \te_{s_m}^{\wb x u v_1 v_2 \cdots v_{m - 1} } ( y_m, v_m ) \\[7pt]
&= \te_V^{\wb x y} (x,u) + \te_I^{\wb x u} (y,v),
\end{align*}
where $y_j = y_{s_j}$, $v_j = v_{s_j}$ for \hbox{$j = 1,2,\ldots,m$}, and so the relation~\eqref{Delta2.2} holds.

To conclude the proof of the sufficiency, it remains to show that the function $\te_V^{\wb x}$ defined by~\eqref{Delta_V-Delta_t} is a transition energy in volume $V$.  If~$V=\Memptyset$ or~$V$ is a singleton, this assertion is trivially true.  Let us now suppose that this assertion is valid for some $V \in \W\setminus\{\Memptyset\}$ and show that it is also valid for~$t \cup V$, where $t \in \mathbb{Z}^d \setminus V$ is arbitrary.  Indeed, for all $x,u,z \in X^t$, $y,v,w \in X^V$ and~$\wb x \in X^{\mathbb{Z}^d \setminus (t \cup V)}$, we can write
\begin{align*}
\te_{t \cup V}^{\wb x} (x y, z w) &+ \te_{t \cup V}^{\wb x} (z w, u v) = \te_t^{\wb x y} (x,z) + \te_V^{\wb x z} (y,w) + \te_V^{\wb x z} (w,v) + \te_t^{\wb x v} (z,u) \\[7pt]
&= \te_t^{\wb x y} (x,z) + \te_V^{\wb x z} (y,v) + \te_t^{\wb x v} (z,u) = \te_t^{\wb x y} (x,z) + \te_{t \cup V}^{\wb x} (z y, u v) \\[7pt]
&= \te_t^{\wb x y} (x,z) + \te_t^{\wb x y} (z,u) + \te_V^{\wb x u} (y,v) = \te_t^{\wb x y} (x,u) + \te_V^{\wb x u} (y,v) = \te_{t \cup V}^{\wb x} (x y, u v).
\end{align*}

Finally, note that the property~\eqref{Delta2.2} of a transition energy field $\Te$ implies the formula~\eqref{Delta_V-Delta_t}, and hence the elements of $\Te$ are uniquely determined by the elements of $\Te^{(1)}$.
\end{proof}

Let us note that the formula~\eqref{Delta_V-Delta_t} expressing the elements of the transition energy field $\Te$ by the elements of the one-point transition energy field $\Te^{(1)}$ is of independent interest too.  In particular, it implies that a transition energy field will be quasilocal if and only if the one-point transition energy field it contains is quasilocal.  Note also that this formula has a clear physical meaning similar to that of the relation~\eqref{Delta2.2}.

The theorem given below establishes a relationship between a one-point transition energy field~$\Te^{(1)}$ and a strictly positive 1--specification $\Qb^{(1)}$.

\begin{Thm}
\label{GibbsForm1-spec}
A family $\Qb^{(1)} = \bigl\{ q_t^{\wb x},\ t \in \mathbb{Z}^d,\, \wb x \in X^{\mathbb{Z}^d \setminus t} \bigr\}$ of functions in lattice points with infinite boundary conditions will be a strictly positive 1--specification if and only if its elements have the Gibbsian form
\[
q_t^{\wb x} (x) = \frac{\exp \{\te_t^{\wb x} (x,u)\}}{\sum \limits_{z \in X^t} {\exp \{\te_t^{\wb x} (z,u)\} }}, \quad x \in X^t,
\]
where $\Te^{(1)} = \bigl\{ \te_t^{\wb x},\ t \in \mathbb{Z}^d,\, \wb x \in X^{\mathbb{Z}^d \setminus t} \bigr\}$ is some one-point transition energy field, and the configuration $u\in X^t$ is arbitrary.  In this case, the one-point transition energy field $\Te^{(1)}$ is uniquely determined.
\end{Thm}

\begin{proof}
Let $\Qb^{(1)} = \bigl\{ q_t^{\wb x},\ t \in \mathbb{Z}^d,\, \wb x \in X^{\mathbb{Z}^d \setminus t} \bigr\}$ be a strictly positive 1--specification.  According to Proposition~\ref{GibbsFormProb}, for each $t \in \mathbb{Z}^d$ and~$\wb x \in X^{\mathbb{Z}^d \setminus t}$, the strictly positive probability distribution $q_t^{\wb x}$ has the required Gibbsian form with
\[
\te_t^{\wb x} (x,u) = \ln \frac{q_t^{\wb x}(x)}{q_t^{\wb x}(u)}, \quad x,u \in X^t,
\]
and $\te_t^{\wb x}$ is a (unique) transition energy in lattice point $t$.  To conclude the proof of the necessity, it remains to show that the family $\Te^{(1)}=\bigl\{ \te_t^{\wb x},\ t \in \mathbb{Z}^d,\, \wb x \in X^{\mathbb{Z}^d \setminus t} \bigr\}$ is a one-point transition energy field, that is, to check that its elements satisfy the consistency conditions~\eqref{DeltaT2}.

By virtue of the conditions~\eqref{Sogl1} of consistency of the elements of~$\Qb^{(1)}$, for all $t,s \in \mathbb{Z}^d$, $x,u \in X^{t}$, $y,v \in X^{s}$ and~$\wb x \in X^{\mathbb{Z}^d \setminus \{ t,s\}}$, we have
\[
\frac{q_t^{\wb x y} (x)}{q_t^{\wb x y} (u)} \, \frac{q_s^{\wb x u} (y)}{q_s^{\wb x u} (v)} = \frac{q_s^{\wb x x} (y)}{q_s^{\wb x x} (v)} \, \frac{q_t^{\wb x v} (x)}{q_t^{\wb x v} (u)},
\]
which yields
\[
\te_t^{\wb x y} (x,u) + \te_s^{\wb x u} (y,v) = \ln \frac{q_t^{\wb x y} (x)}{q_t^{\wb x y} (u)} \, \frac{q_s^{\wb x u} (y)}{q_s^{\wb x u} (v)} = \ln \frac{q_s^{\wb x x} (y)}{q_s^{\wb x x} (v)} \, \frac{q_t^{\wb x v} (x)}{q_t^{\wb x v} (u)} = \te_s^{\wb x x} (y,v) + \te_t^{\wb x v} (x,u).
\]

Let now $\Te^{(1)} = \bigl\{ \te_t^{\wb x},\ t \in \mathbb{Z}^d,\, \wb x \in X^{\mathbb{Z}^d \setminus t} \bigr\}$ be a one-point transition energy field.  According to Proposition~\ref{GibbsFormProb}, for each $t \in \mathbb{Z}^d$ and~$\wb x \in X^{\mathbb{Z}^d \setminus t}$, the function
\[
q_t^{\wb x} (x) = \frac{\exp \{\te_t^{\wb x} (x,u)\}}{\sum \limits_{z \in X^t} {\exp \{\te_t^{\wb x} (z,u)\} }}, \quad x \in X^t,
\]
is a strictly positive probability distribution (not depending on $u \in X^t$).  It remains to show that the elements of the family $\Qb^{(1)}=\bigl\{ q_t^{\wb x},\ t \in \mathbb{Z}^d,\, \wb x \in X^{\mathbb{Z}^d \setminus t} \bigr\}$ satisfy the consistency conditions~\eqref{Sogl1} or, equivalently, that for all $t,s \in \mathbb{Z}^d$, $x,u \in X^{t}$, $y,v \in X^{s}$ and~$\wb x \in X^{\mathbb{Z}^d \setminus \{ t,s\}}$, it holds
\[
\frac{q_t^{\wb x y} (x)}{q_t^{\wb x y} (u)} \, \frac{q_s^{\wb x u} (y)}{q_s^{\wb x u} (v)} = \frac{q_s^{\wb x x} (y)}{q_s^{\wb x x} (v)} \, \frac{q_t^{\wb x v} (x)}{q_t^{\wb x v} (u)}.
\]

The validity of the last relation follows from the chain of equalities
\[
\ln \frac{q_t^{\wb x y} (x)}{q_t^{\wb x y} (u)} \, \frac{q_s^{\wb x u} (y)}{q_s^{\wb x u} (v)} = \te_t^{\wb x y} (x,u) + \te_s^{\wb x u} (y,v) = \te_s^{\wb x x} (y,v) + \te_t^{\wb x v} (x,u) = \ln \frac{q_s^{\wb x x} (y)}{q_s^{\wb x x} (v)} \, \frac{q_t^{\wb x v} (x)}{q_t^{\wb x v} (u)},
\]
which is true by virtue of the conditions~\eqref{DeltaT2} of consistency of the elements of $\Te^{(1)}$.
\end{proof}

It follows from this theorem that any strictly positive 1--specification admits a representation in terms of a one-point transition energy field.  In addition, it is not difficult to see that the 1--specification will be quasilocal if and only if the corresponding one-point transition energy field is quasilocal.

We call a 1--specification \emph{Gibbsian}, if it is strictly positive and quasilocal.  Thus, a strictly positive 1--specification will be Gibbsian if and only if the corresponding one-point transition energy field is quasilocal.

\begin{remark}
The obtained results can be naturally depicted in the following commutative diagram:
\begin{equation}
\label{ComDiag}
\begin{tabular}{ccc}
$\Te^{(1)}$ & $\xlongleftrightarrow{\textrm{~Th.~\ref{Delta1-Delta}~}}$ & $\Te$ \\
\llap{\scriptsize Th.~\ref{GibbsForm1-spec}~}$\updownarrow$ & & $\updownarrow$\rlap{\scriptsize~Th.~\ref{GibbsFormSpec}} \\
$\Qb^{(1)}$ & $\xlongleftrightarrow{\textrm{~Th.~\ref{DN}~}}$ & $\Qb$ \\
\end{tabular}
\end{equation}

Let us emphasize that the chain of transitions $\Qb^{(1)} \leftrightarrow \Te^{(1)} \leftrightarrow \Te \leftrightarrow \Qb$ obtained in the present work provides a new, simpler and physically well-founded proof of Theorem~\ref{DN} (solution of Dobrushin's problem in the strictly positive case).\qed
\end{remark}

\begin{remark}
Using the relation~\eqref{Delta_V-Delta_t}, it is not difficult to obtain the (already mentioned in Section~\ref{Prelim}) formula expressing the elements of the specification~$\Qb$ by the elements of the 1--specification $\Qb^{(1)}$.  Indeed, for all $V \in \W$, $x,u \in X^V$ and $\wb x \in X^{\mathbb{Z}^d \setminus V}$, we can write
\begin{align*}
q_V^{\wb x} (x) &= \frac{\exp \{\te_V^{\wb x} (x,u)\}}{\sum \limits_{z \in X^V} {\exp \{\te_V^{\wb x} (z,u)\} }} \\[7pt]
&= \frac{\exp \{\te_{t_1}^{\wb x x_2 x_3\cdots x_n } ( x_1, u_1 )\} \exp\{ \te_{t_2}^{\wb x u_1 x_3 \cdots x_n} ( x_2, u_2 )\} \cdots \exp\{ \te_{t_n}^{\wb x u_1 u_2 \cdots u_{n - 1} } ( x_n, u_n )\}}{\sum \limits_{z \in X^V} {\exp \{ \te_{t_1}^{\wb x z_2 z_3\cdots z_n } ( z_1, u_1 )\} \exp \{ \te_{t_2}^{\wb x u_1 z_3 \cdots z_n} ( z_2, u_2 )\} \cdots \exp\{ \te_{t_n}^{\wb x u_1 u_2 \cdots u_{n - 1} } ( z_n, u_n ) \} }} \\[7pt]
&= \frac{\displaystyle\frac{q_{t_1}^{\wb x x_2 x_3\cdots x_n } (x_1)}{q_{t_1}^{\wb x x_2 x_3\cdots x_n } (u_1)}\, \frac{q_{t_2}^{\wb x u_1 x_3 \cdots x_n} (x_2)}{q_{t_2}^{\wb x u_1 x_3 \cdots x_n} (u_2)} \cdots \dfrac{q_{t_n}^{\wb x u_1 u_2 \cdots u_{n - 1} } (x_n)}{q_{t_n}^{\wb x u_1 u_2 \cdots u_{n - 1} } (u_n)} }{ \displaystyle\sum \limits_{z \in X^V} \frac{q_{t_1}^{\wb x z_2 z_3\cdots z_n } (z_1)}{q_{t_1}^{\wb x z_2 z_3\cdots z_n } (u_1)} \, \frac{q_{t_2}^{\wb x u_1 z_3 \cdots z_n} (z_2)}{q_{t_2}^{\wb x u_1 z_3 \cdots z_n} (u_2)} \cdots \frac{q_{t_n}^{\wb x u_1 u_2 \cdots u_{n - 1} } (z_n)}{q_{t_n}^{\wb x u_1 u_2 \cdots u_{n - 1} } (u_n)} },
\end{align*}
where $t_1, t_2,\ldots,t_n$ is some enumeration of the elements of the set $V$, and $x_i = x_{t_i}$, $u_i = u_{t_i}$, $z_i = z_{t_i}$ for~$i = 1,2,\ldots,n$.  By virtue of the properties of the elements of $\Te^{(1)}$, this formula does not depend on either the choice of the configuration $u \in X^V$ or the order of enumeration of the elements of~$V$.\qed
\end{remark}

\section{Potential energy}
\label{PE}

Now, basing ourselves on the notion of (one-point) transition energy, we can give a proper mathematical definition of the (one-point) Hamiltonian without involving the notion of potential.

\subsection{One-point Hamiltonian}

Let $\Te^{(1)} = \bigl\{ \te_t^{\wb x},\ t \in \mathbb{Z}^d,\, \wb x \in X^{\mathbb{Z}^d \setminus t} \bigr\}$ be a one-point transition energy field.  Since the elements of~$\Te^{(1)}$ are transition energies in lattice points, they can be represented in the form
\begin{equation}
\label{TEfromPE}
\te_t^{\wb x} (x,u) = \hh_t^{\wb x} (u) - \hh_t^{\wb x} (x), \quad x,u \in X^t,
\end{equation}
where $\hh_t^{\wb x}$ can be interpreted as potential energy in point $t$ under boundary condition~$\wb x$.  Then, the conditions~\eqref{DeltaT2} of consistency of the elements of $\Te^{(1)}$ will be reformulated in the following form: for all $t, s \in {\mathbb{Z}^d}$, $x,u \in X^{t}$ and $y,v \in X^{s}$, it holds
\[
\hh_t^{\wb x y} (u) - \hh_t^{\wb x y} (x) + \hh_s^{\wb x u} (v) - \hh_s^{\wb x u} (y) = \hh_s^{\wb x x} (v) - \hh_s^{\wb x x} (y) + \hh_t^{\wb x v} (u) - \hh_t^{\wb x v} (x),
\]
or, equivalently,
\begin{equation}
\label{PEeq}
\hh_t^{\wb x y} (x) + \hh_s^{\wb x x} (v) + \hh_t^{\wb x v} (u) + \hh_s^{\wb x u} (y) = \hh_s^{\wb x x} (y) + \hh_t^{\wb x y} (u) + \hh_s^{\wb x u} (v) + \hh_t^{\wb x v} (x).
\end{equation}

The converse is also true.  If for a family $\HH^{(1)} = \bigl\{\hh_t^{\wb x},\ t \in \mathbb{Z}^d,\, \wb x \in X^{\mathbb{Z}^d \setminus t} \bigr\}$ of functions in lattice points with infinite boundary conditions, the relation~\eqref{PEeq} holds, the family \hbox{$\Te^{(1)} = \bigl\{ \te_t^{\wb x},\ t \in \mathbb{Z}^d,\, \wb x \in X^{\mathbb{Z}^d \setminus t} \bigr\}$} of functions defined by~\eqref{TEfromPE} is a one-point transition energy field.

These consideration allow us to define the one-point potential energy without using the notion of interaction potential in the following way.

A family $\HH^{(1)} = \bigl\{\hh_t^{\wb x},\ t \in \mathbb{Z}^d,\, \wb x \in X^{\mathbb{Z}^d \setminus t} \bigr\}$ of functions in lattice points with infinite boundary conditions will be called \emph{one-point potential energy} (\emph{one-point Hamiltonian}) if its elements are consistent in the following sense: for all $t, s \in {\mathbb{Z}^d}$, $x,u \in X^{t}$, $y,v \in X^{s}$ and $\wb x \in X^{\mathbb{Z}^d \setminus \{t,s\}}$, the relation~\eqref{PEeq} holds.

Let us note that to each one-point transition energy field $\Te^{(1)}$ correspond a lot of one-point potential energies $\HH^{(1)}$, because each function~$\hh_t^{\wb x}$ is defined up to an additive constant $c_t^{\wb x}$ corresponding to the choice of the zero level of potential energy in the point $t$ under the boundary condition~$\wb x$.  In particular, it is clear that to a quasilocal one-point transition energy field correspond both quasilocal and non-quasilocal one-point potential energies.  On the other hand, the quasilocality of a one-point potential energy ensures the quasilocality of the (unique) corresponding one-point transition energy field.

It is not difficult to see that if $\bPhi$ is a potential, the corresponding one-point Hamiltonian~$\HH_\bPhi^{(1)}$ satisfies the consistency conditions~\eqref{PEeq} and, therefore, is a one-point potential energy.  This justifies the widespread use of sums of interaction potentials as Hamiltonians.

Note also that if $\Qb^{(1)} = \bigl\{ q_t^{\wb x},\ t \in \mathbb{Z}^d,\, \wb x \in X^{\mathbb{Z}^d \setminus t} \bigr\}$ is a strictly positive 1--specification, and if some arbitrary configuration $u^\circ = u^\circ (t,\wb x) \in X^t$ is fixed for each $t \in\mathbb{Z}^d$ and $\wb x \in X^{\mathbb{Z}^d \setminus t}$, the family~$\HH^{(1)} = \bigl\{\hh_t^{\wb x},\ t \in \mathbb{Z}^d,\, \wb x \in X^{\mathbb{Z}^d \setminus t} \bigr\}$ of functions
\[
\hh_t^{\wb x}(x) = -\ln \frac{q_t^{\wb x} (x)}{q_t^{\wb x} (u^\circ)}, \quad x \in X^{t},
\]
is a one-point potential energy.

The following result is an immediate consequence of Theorem~\ref{GibbsForm1-spec}.

\begin{Thm}
\label{J1}
A family $\Qb^{(1)} = \bigl\{ q_t^{\wb x},\ t \in \mathbb{Z}^d,\, \wb x \in X^{\mathbb{Z}^d \setminus t} \bigr\}$ of functions in lattice points with infinite boundary conditions will be a strictly positive 1--specification if and only if its elements have the Gibbsian form
\[
q_t^{\wb x} (x) = \frac{\exp \{ - \hh_t^{\wb x} (x)\}}{\sum \limits_{z \in X^t} {\exp \{ - \hh_t^{\wb x} (z)\} }}, \quad x \in X^t,
\]
where $\HH^{(1)} = \bigl\{\hh_t^{\wb x},\ t \in \mathbb{Z}^d,\, \wb x \in X^{\mathbb{Z}^d \setminus t} \bigr\}$ is some one-point potential energy.
\end{Thm}

In fact, this theorem establishes a representation of a strictly positive 1--specification $\Qb^{(1)}$ in terms of a one-point potential energy~$\HH^{(1)}$.  If, in addition, the one-point potential energy $\HH^{(1)}$ is quasilocal, the 1--specification~$\Qb^{(1)}$ is also quasilocal and, therefore, is Gibbsian.

\begin{remark}
In classical considerations, the properties of the Hamiltonian (corresponding to a given potential) are determined by the properties of the potential.  Among such properties are, for example, the quasilocality, the finite range interaction property, the vacuumness, various symmetries, and so on.  In the framework of our approach, the required properties follow from the corresponding properties of a one-point Hamiltonian (defined by the relation~\eqref{PEeq} without using the notion of potential), while the latters are directly postulated.\qed
\end{remark}

\begin{remark}
Let us note that it is much easier to define the one-point Hamiltonian without using the notion of potential in the continuum case, namely, within the framework of the theory of Gibbs point processes.  Perhaps the first such definition can be found in the unpublished work~\cite{Led1976} of Ledrappier (see also a more recent paper~\cite{NZ1979} of Nguyen and Zessin).  The consistency conditions proposed by Ledrappier in~\cite{Led1976} have the following form:
\[
V_t(\omega)\,V_s(\omega\cup t)=V_s(\omega)\,V_t(\omega\cup s),
\]
where $\omega$ is a configuration (locally-finite subset of $\mathbb{R}^d$), $t,s\in\mathbb{R}^d\setminus\omega$, and $\ln V_t(\omega)$ is interpreted as the energy needed to add the point $t$ to the configuration $\omega$, that is, to move the system from the state $\omega$ to the state $\omega\cup t$.  So, the form and the physical meaning of these consistency conditions are more similar to those of our conditions~\eqref{DeltaT2} of consistency of the elements of the one-point transition energy field, than to those of our conditions~\eqref{PEeq} of consistency of the elements of the one-point potential energy.\qed
\end{remark}

\subsection{Hamiltonian}

Basing ourselves on the definition of the transition energy field, similarly to the one-point case, we can define the \emph{potential energy} (\emph{Hamiltonian}) as a family $\HH = \bigl\{\hh_V^{\wb x},\ V \in \W,\, \wb x \in X^{\mathbb{Z}^d \setminus V} \bigr\}$ of functions satisfying the following consistency conditions: for all $V,I \in \W$ such that $V \cap I = \Memptyset$ and all $x,u \in X^V$, $y \in X^I$ and $\wb x\in X^{\mathbb{Z}^d\setminus (V\cup I)}$, it holds
\begin{equation}
\label{DL-HAM}
\hh_{V \cup I}^{\wb x} (x y) + \hh_V^{\wb x y} (u) = \hh_{V \cup I}^{\wb x} (u y) + \hh_V^{\wb x y} (x). 
\end{equation}
These conditions are based on the relation~\eqref{Delta2} and, as the latter is tantamount to~\eqref{Delta2.2}, are also equivalent to the following consistency conditions: for all $V,I \in \W$ such that $V \cap I = \Memptyset$ and all~$x,u \in X^V$, $y,v \in X^I$ and $\wb x\in X^{\mathbb{Z}^d\setminus (V\cup I)}$, it holds
\[
\hh_{V \cup I}^{\wb x} (x y) + \hh_V^{\wb x y} (u) + \hh_I^{\wb x u} (v) = \hh_{V \cup I}^{\wb x} (u v) + \hh_V^{\wb x y} (x) + \hh_I^{\wb x u} (y).
\]

It is not difficult to see that if $\bPhi$ is a potential, the corresponding Hamiltonian~$\HH_\bPhi$ is a potential energy.

Note also that if $\Qb = \bigl\{ q_V^{\wb x},\ V \in \W,\, \wb x \in X^{\mathbb{Z}^d \setminus V} \bigr\}$ is a strictly positive specification, and if some arbitrary configuration $u^\circ = u^\circ (V,\wb x) \in X^V$ is fixed for each $V \in \W$ and $\wb x \in X^{\mathbb{Z}^d \setminus V}$, the family~$\HH = \bigl\{ \hh_V^{\wb x},\ V \in \W,\, \wb x \in X^{\mathbb{Z}^d \setminus V} \bigr\}$ of functions
\begin{equation}
\label{H_V_ln}
\hh_V^{\wb x} (x) = -\ln \frac{q_V^{\wb x} (x)}{q_V^{\wb x} (u^\circ)}, \quad x \in X^V,
\end{equation}
is a potential energy.

The following result is an immediate consequence of Theorem~\ref{GibbsFormSpec}.

\begin{Thm}
\label{J2}
A family $\Qb = \bigl\{ q_V^{\wb x},\ V \in \W,\, \wb x \in X^{\mathbb{Z}^d \setminus V} \bigr\}$ of functions in finite volumes with infinite boundary conditions will be a strictly positive specification if and only if its elements have the Gibbsian form
\[
q_V^{\wb x} (x) = \frac{\exp \{ - \hh_V^{\wb x} (x)\}}{\sum \limits_{z \in X^V} {\exp \{ - \hh_V^{\wb x} (z)\} }}, \quad x \in X^V,
\]
where $\HH = \bigl\{ \hh_V^{\wb x},\ V \in \W,\, \wb x \in X^{\mathbb{Z}^d \setminus V} \bigr\}$ is some potential energy.
\end{Thm}

In fact, this theorem establishes a representation of a strictly positive specification $\Qb$ in terms of a potential energy~$\HH$.  If, in addition, the potential energy $\HH$ is quasilocal, the specification~$\Qb$ is also quasilocal and, therefore, is Gibbsian.

\begin{remark}
There are many works (see, for example Avertintsev~\cite{Averintsev}, Kozlov~\cite{Kozlov}, Sullivan~\cite{Sullivan-FR,Sullivan}) devoted to the problem of finding conditions under which a specification is Gibbsian (with a potential belonging to some given class).  In all these works, the problem was essentially reduced to proving that the system~$\HH = \bigl\{ \hh_V^{\wb x},\ V \in \W,\, \wb x \in X^{\mathbb{Z}^d \setminus V} \bigr\}$ of function defined by~\eqref{H_V_ln} is a Hamiltonian.  In the absence of a general definition of the Hamiltonian, much effort has gone into proving (mainly using the Mobius inversion formula) the existence of a potential (usually having a complicated form) generating $\HH$.  In the framework of our approach, such considerations are no longer necessary.\qed
\end{remark}

\begin{remark}
We would like to mention that Sullivan introduced in~\cite{Sullivan} the notion of \emph{conditional energy field} as the system of functions~\eqref{H_V_ln} with some particular choice of configurations $u^\circ$ (namely, $u^\circ_t=\theta$ for all~$t\in V$, where $\theta$ is some fixed element of $X$).  This terminology has strongly influenced the choice of the terms \emph{transition energy field} and \emph{one-point transition energy field} introduced in the present work.\qed
\end{remark}

\begin{remark}
Let us note that (as in the one-point case) it is possible to define the Hamiltonian without using the notion of potential also within the framework of the theory of Gibbs point processes. In particular, consistency conditions somewhat similar to our consistency conditions~\eqref{DL-HAM} can be found in the work of Dereudre and Lavancier~\cite{DL2009}.\qed
\end{remark}

In conclusion, we would like to emphasize that Theorems~\ref{J2} and~\ref{J1} show that the representation in Gibbsian form is a necessary and sufficient condition that a family of functions in finite volumes (resp.\ in lattice points) with infinite boundary conditions be a strictly positive specification (resp.\ 1--specification).  This result can, in our opinion, be considered as a justification of the Gibbs formula in the infinite volume case. Moreover, its necessity part answers the long-standing question of how wide the class of specifications, which can be represented in Gibbsian form, is.  This problem was formulated by D.\ Ruelle in the appendix ``Open Problems'' of his book~\cite{R2}.  Our answer is that any strictly positive specification admits Gibbsian representation (with some Hamiltonian).

\section{Towards a general theory of Gibbs random fields based on their energetic description}
\label{GenTheory}

There are many classes of random processes considered in the probability theory.  Usually, the processes belonging to a particular class are characterized by some defining properties of their finite-dimensional or conditional distributions.  Further, the general statements of the corresponding theory are established on the base of these properties, while the applications usually go through some representation theorems expressing the processes in terms of some simple and convenient objects.

The situation is quite different for the class of Gibbs random fields. Historically, instead of being characterized by some properties of their finite-dimensional or conditional distributions, Gibbs random fields have been directly defined by the representation of their conditional distributions in terms of potentials, and only afterwards, the problem of internal characterization of Gibbs random fields was considered.

An introduction to a general theory of Gibbs random fields (which does not use in its initial part the notion of potential) was already presented in the work~\cite{DN2009} of the authors (see also Nahapetian and Khachatryan~\cite{NKh}).  The exposition of the theory was based on the description of random fields by means of 1--specifications.  The notion of the Gibbs random fields was introduced (without using the notion of potential) by means of the following definition: a random field $\PB$ is called \emph{Gibbs random field} if it is strictly positive, the limits
\begin{equation}
\label{CanonicGibbs}
\qb_t^{\wb x} (x) = \lim_{\Lambda \uparrow \mathbb{Z}^d \setminus t} \frac{ \pb_{ t \cup \Lambda } ( x \wb x_\Lambda ) }{ \pb_\Lambda ( \wb x_\Lambda )},\quad x \in X^{t},
\end{equation}
exist for all $t \in \mathbb{Z}^d$ and $\wb x \in X^{\mathbb{Z}^d \setminus t}$, are strictly positive, and the convergence is uniform with respect to $\wb x$.

In this section, we give an alternative presentation of some fragments of the aforementioned theory, basing ourselves on the description of random fields by means of one-point transition energy fields.

Taking into account the considerations of the present work, the above given definition of the Gibbs random field is clearly equivalent to the following one: a random field $\PB$ is called \emph{Gibbs random field} (we also say that the random field~$\PB$ is \emph{Gibbsian}) if it is strictly positive, the limits
\begin{equation}
\label{CanonicDelta}
\tec_t^{\wb x} (x,u) = \lim_{\Lambda \uparrow \mathbb{Z}^d \setminus t} \ln\frac{\pb_{t\cup\Lambda}(x \wb x_\Lambda)}{\pb_{t\cup\Lambda}(u \wb x_\Lambda)}, \quad x,u \in X^t,
\end{equation}
exist for all $t \in \mathbb{Z}^d$ and $\wb x \in X^{\mathbb{Z}^d \setminus t}$, and the convergence is uniform with respect to $\wb x$.

It is easy to see that for a Gibbs random field~$\PB$, the limits~\eqref{CanonicGibbs} form a Gibbsian 1--specification which will be denoted $\QB^{(1)}$ and called \emph{canonical 1--specification} (or \emph{canonical version of one-point conditional distribution}) of the random field~$\PB$, while the limits~\eqref{CanonicDelta} form a quasilocal one-point transition energy field which will be denoted $\Tec^{(1)}$ and called \emph{canonical one-point transition energy field} of the random field~$\PB$.

Also, it is not difficult to show that for a Gibbs random field~$\PB$, the multi-point analogues of the limits~\eqref{CanonicGibbs} and~\eqref{CanonicDelta} exist and form a Gibbsian specification $\QB$ and a quasilocal transition energy field $\Tec$, respectively.  The specification $\QB$ will be called \emph{canonical specification} (or \emph{canonical version of conditional distribution}) of the random field~$\PB$, while the transition energy field $\Tec$ will be called \emph{canonical transition energy field} of the random field~$\PB$.

Note that the four canonical objects $\Tec^{(1)}$, $\QB^{(1)}$, $\Tec$ and $\QB$ corresponding to a Gibbs random field $\PB$ uniquely determine one another according to the commutative diagram~\eqref{ComDiag}.  So, the assertions presented below concerning any one of them will, of course, admit analogues for the other three.

Let $\Te^{(1)} = \bigl\{ \te_t^{\wb x},\ t \in \mathbb{Z}^d,\, \wb x \in X^{\mathbb{Z}^d \setminus t} \bigr\}$ be a one-point transition energy field.  We will say that a random field $\PB$ is \emph{compatible with (the one-point transition energy field)} $\Te^{(1)}$ if for any $t\in\mathbb{Z}^d$, the limits~\eqref{CanonicDelta} satisfy $\tec_t^{\wb x} = \te_t^{\wb x}$ for almost all (with respect to the measure $\PB$) configurations $\wb x \in X^{\mathbb{Z}^d\setminus t}$.  The definition of the compatibility of a random field with a transition energy field is similar.

As it was already said, the canonical one-point transition energy field $\Tec^{(1)}$ of a Gibbs random field $\PB$ is quasilocal.  The converse is also true: if a random field $\PB$ is compatible with a quasilocal one-point transition energy filed, the random field $\PB$ is Gibbsian (see, for example, the work~\cite{DN2009} of the authors, where the analogues of this assertion in terms of specifications and 1-specifications can be found).

Further, let us note that with the given definition of the Gibbs random field, the question of its existence does not arise, because strictly positive Markov random fields are obviously Gibbsian.  Moreover, it can be shown that the set of Gibbs random fields is a uniform extension of the set of strictly positive Markov random fields.

It should also be noted that not all the strictly positive random fields are Gibbsian.  This can be seen from any one from the following two examples.

\begin{example}
Consider the random field $\PB$ with state space $X = \{ 0,1 \}$ such that for any $V \in \W$, the finite-dimensional distribution of $\PB$ in the volume $V$ is given by
\[
\pb_V (x) = \frac{1}{(\bil| V \bir| + 1) C_{\bil| V \bir|}^{\bil| x \bir|}}, \quad x \in X^V.
\]

For all $t \in \mathbb{Z}^d$, $\Lambda \in \W(\mathbb{Z}^d \setminus t)$ and $\wb x \in X^{\mathbb{Z}^d \setminus t}$, we have
\[
\frac{\pb_{ t \cup \Lambda} (1_t \wb x_\Lambda)}{ \pb_{t \cup \Lambda} (0_t\wb x_\Lambda)} = \frac{C_{\bil| \Lambda \bir| + 1}^{\bil| \wb x_\Lambda \bir|}}{C_{\bil| \Lambda \bir| + 1}^{\bil| \wb x_\Lambda \bir| + 1}} = \frac{\bil|\wb x_\Lambda \bir| + 1}{\bil| \Lambda \bir| - \bil|\wb x_\Lambda \bir| + 1}.
\]
Denote $U_t^p$, where $t \in \mathbb{Z}^d$ and $p \in [ 0,1 ]$, the set of all configurations $\wb x \in X^{\mathbb{Z}^d \setminus t}$ such that the limit~$\lim \limits_{\Lambda \uparrow \mathbb{Z}^d \setminus t} \frac{\bil|\wb x_\Lambda \bir|}{\bil| \Lambda \bir|}$ exists and is equal to $p$.  Clearly, if $\wb x \in U_t^p$ with some $p\in(0,1)$, the limits~\eqref{CanonicDelta} exist and are given by~$\tec_t^{\wb x} (1_t,0_t) = -\tec_t^{\wb x} (0_t,1_t) = \ln\frac{p}{1-p}$ $\bigl($and $\tec_t^{\wb x} (1_t,1_t) = \tec_t^{\wb x} (0_t,0_t) = 0\bigr)$.  On the contrary, for all~$\wb x \in X^{\mathbb{Z}^d \setminus t} \Bigsetminus \bigcup \limits_{p \in (0,1)} U_t^p$, the limits~\eqref{CanonicDelta} do not exist (for $x\ne u$), and so the random field~$\PB$ is not Gibbsian.\qed
\end{example}

\begin{example}
Let $\alpha \in (0,1)$, and let $\PB^{(p_1)}$ and $\PB^{(p_2)}$ be the Bernoulli random fields with parameters~$p_1$ and $p_2$, respectively, where $p_1,p_2\in(0,1)$ are such that $p_1 \neq p_2$.  Let us show that the convex combination (mixture) $\PB=\alpha \PB^{(p_1)} + (1-\alpha) \PB^{(p_2)}$ of $\PB^{(p_1)}$ and $\PB^{(p_2)}$ is not Gibbsian.

For all $t \in \mathbb{Z}^d$, $\Lambda \in \W (\mathbb{Z}^d \setminus t )$ and $\wb x \in X^{\mathbb{Z}^d \setminus t}$, we can write
\[
\frac{\pb_{t \cup \Lambda} (1_t \wb x_\Lambda)}{ \pb_{t \cup \Lambda} (0_t \wb x_\Lambda)} = \frac{ \alpha \, p_1 + \wb \alpha \, p_2 \exp \{ \bil| \Lambda \bir| f_\Lambda (\wb x_\Lambda) \} }{ \alpha \, \wb p_1 + \wb \alpha \, \wb p_2 \exp \{ \bil| \Lambda \bir| f_\Lambda (\wb x_\Lambda) \} },
\]
where $\wb\alpha=1-\alpha$, $\wb p_1=1-p_1$, $\wb p_2=1-p_2$ and
\[
f_\Lambda (\wb x_\Lambda) = \frac{\bil| \wb x_\Lambda \bir|}{\bil| \Lambda \bir|} \ln \frac{p_2}{p_1} + \biggl( 1 - \frac{\bil| \wb x_\Lambda \bir|}{\bil| \Lambda \bir|} \biggr) \ln \frac{\wb p_2}{\wb p_1}.
\]
It is not difficult to see that there exist a configuration $\wb y^\circ$ such that
\[
\liminf_{\Lambda \uparrow \mathbb{Z}^d \setminus t} \frac{\bil| \wb y^\circ_\Lambda \bir| }{\bil| \Lambda \bir| } =0 \quad \text{and} \quad
\limsup_{\Lambda \uparrow \mathbb{Z}^d \setminus t} \frac{\bil| \wb y^\circ_\Lambda \bir| }{\bil| \Lambda \bir| } =1.
\]
Since the quantities $\ln \dfrac{p_2}{p_1}$ and $\ln \dfrac{\wb p_2}{\wb p_1}$ are of opposite signs, we have
\[
\liminf_{\Lambda \uparrow \mathbb{Z}^d \setminus t} \bil| \Lambda \bir| f_\Lambda (\wb y^\circ_\Lambda) = - \infty \quad \text{and} \quad
\limsup_{\Lambda \uparrow \mathbb{Z}^d \setminus t} \bil| \Lambda \bir| f_\Lambda (\wb y^\circ_\Lambda) = + \infty.
\]
Thus, the limits~\eqref{CanonicDelta} do not exist for $\wb x = \wb y^\circ$ (and $x\ne u$), and so the random field $\PB$ is not Gibbsian.\qed
\end{example}

Note that there are a lot of other examples of random fields that are not Gibbsian (see, for example, van Enter, Fern\'andez and Sokal~\cite{EFS}).  So, the set of Gibbs random fields is a proper subset of the set of strictly positive random fields.  At the same time, it is not difficult to prove the following statement.

\begin{Prop}
The set of Gibbs random fields is dense in the set of strictly positive random fields with respect to the topology of weak convergence.
\end{Prop}

Note also that, as shows the following example (see also Georgi~\cite{Georg}), there may exist different Gibbs random fields having the same canonical one-point transition energy field $\Tec^{(1)}$ (and hence the same $\QB^{(1)}$, $\Tec$ and $\QB$).

\begin{example}
Let $\mathbb{N} = \{1,2,3,\ldots\}$, and let the numbers $c_j$, $j \in \mathbb{N}$, be such that $0 < c_j < 1$ and~$\prod \limits_{j=1}^{\infty} c_j >0$.  For $t \in \mathbb{N}$, we put $k_t = \prod \limits_{j=t}^\infty c_j$.  Let $\PB^{+}$ and $\PB^{-}$ be the Markov chains on $\mathbb{N}$ with the same state space~$X = \{ -1,1 \}$, whose initial distributions and transition probabilities are given~by
\[
p_1^{\pm} (x_1) = \frac{1 \pm x_1 k_1}{2} \quad\text{and}\quad
p_t^{\pm} (x_{t-1}, x_t) = \frac{1 + c_{t-1} x_{t-1} x_t}{2} \, \frac{1 \pm x_t k_t}{1 \pm x_{t-1} k_{t-1}},
\]
where $t \in \mathbb{N}$ and $x_t \in X^t$.  Let us show that the random fields $\PB^{+}$ and $\PB^{-}$ have the same canonical transition energy field.

First, let us note that for all $n \in \mathbb{N}$, the finite-dimensional distributions $\pb_{\{1,2,\ldots,n\}}^+$ and~$\pb_{\{1,2,\ldots,n\}}^-$ of the random fields $\PB^{+}$ and $\PB^{-}$ have the form
\[
\pb_{\{1,2,\ldots,n\}}^\pm (x_1 x_2 \cdots x_n) = \Biggl( \prod_{j=1}^{n - 1} \frac{1 + c_j x_j x_{j+1} }{2} \Biggr)\, \frac{1 \pm x_n k_n}{2}, \quad x_j \in X^j,\ j = 1,\ldots,n.
\]

Further, let us note that for all $t\in \mathbb{N}\setminus 1$ and for $\Lambda = \{1,2,\ldots,t-1,t+1,\ldots,t+n\}$, $n \in \mathbb{N}$, we have
\begin{equation}
\label{usl_gen}
\frac{\pb_{t \cup \Lambda}^{+} (x_t y)}{\pb_{t \cup \Lambda}^{+} (u_t y) } = \frac{\pb_{t \cup \Lambda}^{-} (x_t y)}{\pb_{t \cup \Lambda}^{-} (u_t y) } = \frac{(1 + c_{t-1} y_{t-1} x_t )(1+ c_t x_t y_{t+1})}{(1 + c_{t-1} y_{t-1} u_t )(1+ c_t u_t y_{t+1})},
\end{equation}
where $x_t,u_t \in X^t$ and $y \in X^\Lambda$.  Indeed, this follows immediately from the relation
\[
\pb_{t \cup \Lambda}^{\pm} (x_t y) = \Biggl(\prod_{j=1}^{t - 2} \frac{ 1 + c_j y_j y_{j+1} }{2} \Biggr) \, \frac{1 + c_{t-1} y_{t-1} x_t}{2} \, \frac{1 + c_t x_t y_{t+1} }{2} \, \Biggl(\prod_{j=t+1}^{t+n-1} \frac{1 + c_j y_j y_{j+1} }{2} \Biggr) \, \frac{1 \pm y_{t+n} k_{t+n}}{2}.
\]

It is not difficult to see that~\eqref{usl_gen} holds also for all $\Lambda\in \W(\mathbb{N}\setminus t)$ containing the bounda\-ry~$\{t-1,t+1\}$ of the point $t$.

For $t = 1$, we similarly obtain
\begin{equation}
\label{usl_gen_t1}
\frac{\pb_{1 \cup \Lambda}^{+} (x_1 y)}{\pb_{1 \cup \Lambda}^{+} (u_1 y) } = \frac{\pb_{1 \cup \Lambda}^{-} (x_1 y)}{\pb_{1 \cup \Lambda}^{-} (u_1 y) } = \frac{1 + c_1 x_1 y_2}{1 + c_1 u_1 y_2}, \quad x_1,u_1 \in X^1,
\end{equation}
for all $\Lambda\in \W(\mathbb{N}\setminus 1)$ such that $2\in \Lambda$ and all $y \in X^\Lambda$.

It follows from~\eqref{usl_gen} and~\eqref{usl_gen_t1} that the limits
\[
\tec_t^{\wb x}(x_t,u_t) = \lim_{\Lambda \uparrow \mathbb{N} \setminus t} \ln\frac{\pb_{t \cup \Lambda}^{\pm} (x_t \wb x_\Lambda)}{\pb_{t \cup \Lambda}^{\pm} (u_t \wb x_\Lambda)} =
\begin{cases}
\ln \dfrac{(1 + c_{t - 1} \wb x_{t -1} x_t )(1+ c_t x_t \wb x_{t+1})}{(1 + c_{t - 1} \wb x_{t -1} u_t )(1+ c_t u_t \wb x_{t+1})}, & \!\text{if }t\in \mathbb{N}\setminus 1,\\[15pt]
\ln \dfrac{1 + c_1 x_1 \wb x_2}{1 + c_1 u_1 \wb x_2}, & \!\text{if }t=1,
\end{cases}
\quad x_t,u_t\in X^t,
\]
exist, and the convergence is uniform with respect to $\wb x \in X^{\mathbb{N} \setminus t}$.  So, the functions $\tec_t^{\wb x}$ form the (common) canonical one-point transition energy field of the random fields~$\PB^{+}$ and $\PB^{-}$.\qed
\end{example}

In the theory of Gibbs random fields, a special attention is payed to the problems of existence and of uniqueness of a random field compatible with a given specification (see, for example, Dobrushin~\cite{D1, D4}, Georgii~\cite{Georg}).  One of the key results in this area (reformulated in terms of one-point transition energy fields) is given by the following theorem (see the paper~\cite{D1} of Dobrushin for the formulation in terms of specifications and the works~\cite{DN1998, DN2001, DN2004, DN2009} of the authors for the formulation in terms of~\hbox{1--specifications}).

\begin{Thm}
\label{ExUniqDobr}
Let $\Te^{(1)} = \bigl\{ \te_t^{\wb x},\ t \in \mathbb{Z}^d,\, \wb x \in X^{\mathbb{Z}^d \setminus t} \bigr\}$ be a one-point transition energy field.  If $\Te^{(1)}$ is quasilocal, there exists a random field~$\PB$ compatible with $\Te^{(1)}$.  Moreover, the random field $\PB$ is Gibbsian and~$\Te^{(1)}$ is its canonical one-point transition energy field.

If, in addition, the Dobrushin uniqueness condition is satisfied, that is,
\[
\sup_{t \in \mathbb{Z}^d}\sum_{s \in \mathbb{Z}^d \setminus t} \sup\ \frac{1}{2} \sum_{x \in X^t} \,\Biggl|\, \frac{\exp \{\te_t^{\wb x} (x,u)\}}{\sum \limits_{z \in X^t} {\exp \{\te_t^{\wb x} (z,u)\} }} - \frac{\exp \{\te_t^{\wb y} (x,u)\}}{\sum \limits_{z \in X^t} {\exp \{\te_t^{\wb y} (z,u)\} }} \,\Biggr|\, < 1,
\]
where $u\in X^t$ is  arbitrary and the second $\sup$ is taken over all pairs of configurations $\wb x,\wb y \in X^{\mathbb{Z}^d \setminus t}$ differing only in the point~$s$, the random field $\PB$ is the unique random field compatible with~$\Te^{(1)}$.
\end{Thm}

Note that though the phrasing of the above theorem is based on the one-point transition energy field $\Te^{(1)}$, as a matter of fact, it uses the elements of the corresponding 1--specification $\Qb^{(1)}$ in the formulation of the Dobrushin uniqueness condition.  Below we give an alternative theorem formulated purely in terms of one-point transition energy fields, which however uses a sightly stronger uniqueness condition.

\begin{Thm}
Let $\Te^{(1)} = \bigl\{ \te_t^{\wb x},\ t \in \mathbb{Z}^d,\, \wb x \in X^{\mathbb{Z}^d \setminus t} \bigr\}$ be a one-point transition energy field.  If $\Te^{(1)}$ is quasilocal, there exists a random field~$\PB$ compatible with $\Te^{(1)}$.  Moreover, the random field $\PB$ is Gibbsian and~$\Te^{(1)}$ is its canonical one-point transition energy field.

If, in addition,
\[
\sup_{t \in \mathbb{Z}^d}\sum_{s \in \mathbb{Z}^d \setminus t} \sup\ \frac{1}{2} \max_{x,y \in X^t} \bigl| \te_t^{\wb x} (x,y) - \te_t^{\wb y} (x,y) \bigr| < 1,
\]
where the second $\sup$ is taken over all pairs of configurations $\wb x,\wb y \in X^{\mathbb{Z}^d \setminus t}$ differing only in the point $s$, the random field $\PB$ is the unique random field compatible with $\Te^{(1)}$.
\end{Thm}

This theorem is an immediate consequence of Theorem~\ref{ExUniqDobr} and of the inequality
\[
\sum_{x \in X^t} \bigl| q_t^{\wb x} (x) - q_t^{\wb y} (x) \bigr| \leq \max_{x,y \in X^t} \bigl| \te_t^{\wb x} (x,y) - \te_t^{\wb y} (x,y) \bigr|,
\]
which is essentially contained in the proof of Theorem 6.35 of Friedli and Velenik~\cite{FV2017} (see also Nahapetian and Khachatryan~\cite{NKh}).

As we have seen above, the mixtures and the limits of Gibbs random fields are not always Gibbsian.  However, this is no longer the case if we consider Gibbs random fields having the same canonical one-point transition energy field.  Moreover, the following theorem describing the structure of the set of random fields compatible with a given quasilocal one-point transition energy field holds (see the paper~\cite{D2} of Dobrushin for the formulation in terms of specifications and the work~\cite{DN2009} of the authors for the formulation in terms of 1--specifications).

\begin{Thm}
Let $\Te^{(1)}$ be a quasilocal one-point transition energy field.  The set $\cG \bigl(\Te^{(1)}\bigr)$ of random fields compatible with $\Te^{(1)}$ coincides with the set of Gibbs random fields having $\Te^{(1)}$ as canonical one-point transition energy field.  Furthermore, this set is convex and closed.
\end{Thm}

In conclusion, let us note that we presented only some basic statements of the theory of Gibbs random fields.  However, many other results of the Gibbsian theory can also be stated in terms of one-point transition energy fields.  The so-obtained description of random fields establishes a straightforward relationship between the probabilistic notion of (Gibbs) random field and the physical notion of (transition) energy, and so opens the possibility to directly apply probabilistic methods to such mathematical problems of statistical physics as, for example, decay of correlations, limit theorems, cluster expansions, and so on.

As we have already noted, the construction of the theory of random fields begins with the establishment of general facts in terms of finite-dimensional distributions of a random field. Limit theorems are also included in the area of these questions.  Already Dobrushin gave in~\cite{D1} an estimate of the mixing coefficient of the random field in terms of its one-point conditional distributions. Some further results in this direction can be found in Dalalyan and Nahapetian~\cite{DalN2011}.  Having the mixing property, one can indicate the conditions for the validity of limit theorems without appealing to the notion of potential.

Note also that besides systematizing and simplifying the Gibbsian theory by giving it a more traditional form, the axiomatics we propose allows to expand the scope of its applications.  Here are some examples.

The two species formulation of the Widom-Rowlinson model is based on a potential and is defined in both the continuum (see Widom and Rowlinson~\cite{WR}) and lattice (see Lebowitz and Gallavotti~\cite{LG}) cases.  In contrary, the area-interaction formulation of the Widom-Rowlinson model was defined in the continuum case only, directly through the Hamiltonian.  Our axiomatics makes it possible to define the lattice analogue of the area-interaction model.  For example, the corresponding one-point transition energy field can be directly defined by
\[
\te_t^{\overline x}(0,1) = -\te_t^{\overline x}(1,0) = \alpha + \beta \mathop{\mathrm{Card}}\biggl(B(t,R) \Big\backslash \bigcup_{s \,:\, \overline  x_s=1}B(s,R)\biggr),\quad t\in\mathbb{Z}^d,\ \overline x\in X^{\mathbb{Z}^d\setminus t},
\]
where $X=\{0,1\}$, and $B(t,R)$ denotes a ball of radius $R\in[1,+\infty)$ centered in $t$.

The proposed axiomatics can be easily extended to the case $X=\mathbb{R}^1$, which allows, for example, to consider Gaussian random fields. For a Gaussian random field, the elements of the one-point transition energy field have the form
\[
\te_t^{\overline x} (x,u) = \frac{a_{tt}}{2} \bigl( (u-m_t)^2 - (x-m_t)^2 \bigr) + (u-x) \sum \limits_{r \in \mathbb{Z}^d \backslash \{t\}} a_{tr}(\overline x_r - m_r), \quad x,u \in \mathbb{R},
\]
where $A = \{a_{ts}, t,s \in \mathbb{Z}^d\}$ is the matrix inverse to the covariance matrix of the Gaussian field, and $m = \{m_t, t \in \mathbb{Z}^d\}$ is the vector of its mean values.  Using this expression, one can extend the known results of Dobrushin~\cite{D5} and Georgii~\cite{Georg} about Gibssian representation of Gaussian random fields to the case of infinite-range potentials (a paper on the subject is in preparation).

\section*{Acknowledgments}

The authors are grateful to Sergey Pirogov for useful discussions and comments.  The authors would also like to thank the reviewer for careful reading of the paper and for his useful remarks.  Serguei Dachian acknowledges support from the Labex CEMPI (ANR-11-LABX-0007-01).

\end{document}